\documentclass[twocolumn,twoside]{IEEEtran}
\usepackage{amsmath,amssymb,amsthm,epsfig,color,graphicx,epstopdf}
\usepackage{enumerate,url,algpseudocode,algorithm,wasysym,balance}
\usepackage[table]{xcolor}
\usepackage{accents}
\usepackage[utf8]{inputenc}

\usepackage{bm,amsfonts}

\DeclareMathOperator{\subjectto}{subject~to}

\DeclareMathOperator*{\minimize}{minimize}

\usepackage[hidelinks]{hyperref}       

\newtheorem{lemma}{Lemma}
\newtheorem{corollary}{Corollary}
\newtheorem{theorem}{Theorem}

\theoremstyle{remark}\newtheorem{remark}{Remark}

\def\b{{\mathbf b}}
\def\c{{\mathbf c}}

\def\p{{\mathbf p}}

\def\s{{\mathbf s}}

\def\w{{\mathbf w}}
\def\x{{\mathbf x}}
\def\y{{\mathbf y}}
\def\z{{\mathbf z}}

\def\A{{\mathbf A}}

\begin{document}
\title{Randomized Block Frank-Wolfe for Convergent Large-Scale Learning}

\author{
 Liang Zhang,~\IEEEmembership{Student Member,~IEEE,}	 Gang Wang,~\IEEEmembership{Student Member,~IEEE,}\\
 Daniel Romero,~\IEEEmembership{Member,~IEEE}, and 
	 Georgios B. Giannakis,~\IEEEmembership{Fellow,~IEEE}

\thanks{
	
	L. Zhang, G. Wang, and G. B. Giannakis are with the Digital Technology Center and the Department of Electrical and Computer Engineering at the University of Minnesota, Minneapolis, MN 55455, USA. D. Romero is with the Department of Information and Communication Technology, University of Agder, Grimstad 4879, Norway. G. Wang is also with the State Key Laboratory of Intelligent Control and Decision of Complex Systems, Beijing Institute of Technology, Beijing 100081, P. R. China. E-mails: \{zhan3523,\,gangwang,\,georgios\}@umn.edu, daniel.romero@uia.no.} 

}

\maketitle

\begin{abstract}
Owing to their low-complexity iterations, Frank-Wolfe (FW) solvers are well suited for various large-scale learning tasks. When block-separable constraints are present, randomized block FW (RB-FW) has been shown to further reduce complexity by  updating only a fraction of coordinate blocks per iteration. To circumvent the limitations of existing methods, the present work develops step sizes for RB-FW that enable a flexible selection of the number of blocks to update per iteration while ensuring convergence and feasibility of the iterates. To this end,
convergence rates of RB-FW are established through computational bounds on a primal sub-optimality measure and on the duality gap. 
The novel bounds extend the existing convergence analysis, which only applies to a step-size sequence that does not generally lead to feasible iterates. Furthermore, two classes of step-size sequences that guarantee feasibility of the iterates are also proposed to enhance flexibility in choosing decay rates. 
The novel convergence results are markedly broadened to encompass also nonconvex objectives, and further assert that RB-FW with exact line-search reaches a stationary point at rate $\mathcal{O}(1/\sqrt{t})$.  
Performance of RB-FW with different step sizes and number of  blocks is demonstrated in two applications, namely charging of electrical vehicles and structural support vector machines. Extensive simulated tests
demonstrate the performance improvement of RB-FW relative to existing randomized single-block FW methods.

\end{abstract}


\section{Introduction}\label{sec:intro}
The Frank-Wolfe (FW) algorithm~\cite{Frankwolfe}, also known as conditional gradient descent~\cite{cgd1970}, has well-documented merits as a  first-order solver especially for smooth constrained optimization tasks over convex compact sets. FW has recently received  revived  interest due to its simplicity and versatility in handling  structured constraint sets in various signal processing and machine learning applications~\cite{jaggi2013revisiting}. This growing popularity is due to its per-iteration simplicity that only entails minimizing a linear function over the feasible set, whereas competing first-order alternatives, such as projected gradient descent~\cite{Be99} and their accelerated versions~\cite{Nesterov}, involve minimizing a quadratic function over the feasible set per iteration. Typically, solving a constrained linear optimization is considerably easier than finding the aforementioned projections per iteration. 
The resulting savings benefit diverse large-scale learning tasks, 
including matrix completion~\cite{jaggi2010simple}, multi-class classification~\cite{harchaoui2015conditional}, image reconstruction~\cite{harchaoui2015conditional},  structural support vector machines (SVMs)~\cite{lacoste2013block}, particle filtering~\cite{lacoste2015sequential}, sparse phase retrieval \cite{sparta,taf}, and scheduling electric vehicle (EV) charging~\cite{liang16scalable}.





Despite its simplicity, FW can become prohibitively expensive when dealing with high-dimensional data. For this reason, \emph{randomized} single-block FW has been advocated for solving large-scale convex constrained programs~\cite{lacoste2013block}, where only a randomly selected block of variables is updated per iteration. At the price of obtaining the duality gap, convergence of randomized single-block FW has been improved in~\cite{osokin2016minding}. Furthermore, randomized multiple-block FW was devised to reduce convergence time by updating multiple blocks per iteration in \emph{parallel}~\cite{stofw}. Unfortunately, feasibility of the resulting iterates is in general not guaranteed by the original parallel randomized block (RB)-FW~\cite{stofw}. Moreover, all results on randomized FW focus on convex objectives, and convergence of RB-FW for nonconvex programs remained hitherto an~open problem.
 

The present paper is the first to introduce a broad class of step sizes for RB-FW that offer: (i) guaranteed convergence and feasibility of the iterates along with (ii) flexibility to select a step-size sequence whose decay rate is attuned to the problem at hand. RB-FW with this rich class of step sizes subsumes the classical FW as well as the randomized single-block FW solvers as special cases. We further broaden the scope of RB-FW by allowing for nonconvex smooth objective functions. 
Specifically, we establish that RB-FW with typical step sizes attains a stationary point at rate $\mathcal{O}(1/\log t)$, whereas line-search-based step sizes enjoys an improved rate of $\mathcal{O}({1}/{\sqrt{t}})$. Remarkably, the latter coincides with the rate afforded by  classical FW for nonconvex problems~\cite{lacoste2016convergence}. Finally, simulated tests on optimal coordination of EV charging and structural SVMs corroborate the merits of RB-FW with our novel step sizes relative to  single-block FW.

The remainder of this paper is organized as follows. Section~\ref{sec:prefw} outlines the FW and RB-FW algorithms. Section~\ref{sec:bcfw} describes two novel families of step sizes for RB-FW, and establishes their feasibility and convergence. Section~\ref{sec:bcfwnoncvx} derives the RB-FW convergence rates for non-convex programs, whereas 
Section~\ref{sec:fw} highlights the implications of the results in Section~\ref{sec:bcfw} for classical FW. Section~\ref{sec:app} shows the merits of RB-FW in two application settings, whereas Section~\ref{sec:tests} tests the RB-FW performance numerically. Finally, Section~\ref{sec:conclusions} concludes~the~paper.

Regarding common notation, lower- (upper-) case boldface letters represent column vectors (matrices). Sets are denoted by calligraphic letters,  $|\mathcal{B}|$ stands for the cardinality of set $\mathcal{B}$, and
$\mathcal{N}\setminus \mathcal{B} := \{x\in \mathcal{N}: x \notin \mathcal{B}\}$   denotes set difference. 
 Symbol $^{\top}$  is reserved for transposition of vectors and matrices, whereas $\mathbf{0}$ and $\mathbf{1}$ denote the all-zero and all-one vectors of suitable dimensions, respectively. Operator $\lceil x \rceil$ gives the smallest integer greater than or equal to $x$, and $\log(x)$ returns the natural logarithm of $x$. 

\color{black}
\section{Preliminaries}\label{sec:prefw}
The classical FW algorithm~\cite{Frankwolfe} aims at solving  the following generic constrained optimization problem
\begin{align}\label{eq:problem}
\minimize_{\mathbf{x}\in\mathbb{R}^d}\quad &~~ f(\mathbf{x})\\
\subjectto \quad & ~~ \mathbf{x}\in \mathcal{X}\nonumber
\end{align}
where  $f(\x)$ is  convex and differentiable, while the feasible set $\mathcal{X}$ is convex and compact. A  number of problems in signal processing and machine learning, e.g., ridge regression or basis pursuit~\cite{chen2001atomic}, can be expressed in this form.
Listed as Algorithm~\ref{alg:fw}, FW is initialized with a feasible $\x^0$. Given iterate $\mathbf{x}^t$, it then solves the following so-termed ``linear oracle''
\begin{equation}\label{eq:lo}
\mathbf{s}^t:= \arg\min_{\mathbf{s} \in \mathcal{X}} ~ \mathbf{s}^{\top} \nabla f(\mathbf{x}^t)
\end{equation}	
and uses a convex combination of $\mathbf{s}^t$ with $\mathbf{x}^t$ to obtain 	
\begin{equation}
\x^{t+1} = (1-\gamma_t)\x^t+\gamma_t \s^t
\end{equation}
where the step size $\gamma_t\in (0,1]$ is typically selected as~\cite{jaggi2013revisiting} 
\begin{equation}\label{eq:typicalr}
\gamma_t = \frac{2}{t+2} \;.
\end{equation}
Alternatively, $\gamma_t$ can be chosen via line search, which picks $\x^{t+1}$ as
the best point on the line segment between $\x^t$ and $\s^t$:
\begin{equation} \label{eq:linesearch}
\gamma_t  = \arg \min_{0\leq \gamma \leq 1} f\left((1-\gamma)\mathbf{x}^t+\gamma\mathbf{s}^t\right).
\end{equation}
In either case,  Algorithm~\ref{alg:fw} converges at rate $\mathcal{O}(1/t)$~\cite{jaggi2013revisiting}. 

\begin{algorithm}[t]
\caption{Frank-Wolfe \cite{Frankwolfe}}\label{alg:fw} 
\begin{algorithmic}[1]
\renewcommand{\algorithmicrequire}{\textbf{Input:}}
\renewcommand{\algorithmicensure}{\textbf{Output:}}
\State Initialize  $t=0$,  $\mathbf{x}^0\in \mathcal{X}$
\While {stopping criterion not met}
\State Compute 
$\mathbf{s}^t= \arg\min_{\mathbf{s} \in \mathcal{X}} ~ \mathbf{s}^{\top} \nabla f(\mathbf{x}^t)$
\State Update $\mathbf{x}^{t+1}=(1-\gamma_t)\mathbf{x}^t+\gamma_t\mathbf{s}^t$
\State $t \leftarrow t+1$
\EndWhile
\end{algorithmic}
\end{algorithm}

When $d$ is large, updating all $d$ entries of $\x$ at each $t$ is computationally challenging. Randomized FW alleviates this difficulty by updating only a subset of the $d$ entries~\cite{lacoste2013block},~\cite{stofw}. Splitting $\mathbf{x}$ into $N_b$ blocks $\{\mathbf{x}_n\}_{n=1}^{N_b}$ with respective feasible sets $\{\mathcal{X}_n\}_{n=1}^{N_b}$ assumed convex and compact, 
\eqref{eq:problem} becomes
\begin{align}\label{eq:blockproblem}
	\minimize_{\mathbf{x}\in\mathbb{R}^d} \quad & f(\mathbf{x})\\
	\subjectto \quad & \mathbf{x}_1\in \mathcal{X}_1,\, \ldots,\, \mathbf{x}_{N_b} \in \mathcal{X}_{N_b}\nonumber
\end{align}
where $\x^{\top} := [\x_1^\top, \x_2^\top, \cdots, \x_{N_b}^\top]$. Note that if $N_b = 1$, then \eqref{eq:blockproblem} boils down to \eqref{eq:problem}.

The decomposition $\mathcal{X}=\mathcal{X}_1 \times \ldots \times \mathcal{X}_{N_b}$ entails no loss of generality since any $\mathcal{X}$ can be expressed in this form by setting $N_b=1$. It also emerges naturally in a number of applications, including the dual problem of structural SVMs~\cite{lacoste2013block}, trace-norm regularized tensor completion~\cite{liu2013tensor},  EV charging ~\cite{liang16scalable}, the dual problem of group fused Lasso~\cite{alaiz2013group}, and structured sub-modular minimization~\cite{jegelka2013reflection}. Thanks to the separable structure of $\mathcal{X}$,
the linear oracle in~\eqref{eq:lo} decouples across  $N_b$ blocks as  
\begin{equation}\label{eq:distlinear}
	\s_n^t = \arg\min_{\mathbf{s}_n\in \mathcal{X}_n}~\langle \s_n, \nabla_{\x_n} f(\x^t) \rangle, \quad n = 1, 2, \ldots, N_b 
\end{equation}
where $\nabla_{\x_n} f(\x^t)$ comprises the partial derivatives of $f(\x)$ with respect to the entries of  $\x_n$. 

Instead of solving the $N_b$ problems in \eqref{eq:distlinear}, RB-FW reduces complexity by solving just $B$ of them, where $B \in \{1, \ldots, N_b\}$ is a pre-selected constant.
Let $\mathcal{N}_b :=\{1, \ldots, N_b\}$ be the index set of all blocks,   and let  $\mathcal{B}_t$ be chosen at iteration $t$ uniformly at random among all subsets of $\mathcal{N}_b$ with $B$ elements. The RB-FW solver of~\eqref{eq:blockproblem} is summarized as Algorithm~\ref{alg:bcfw}.  To save computation time,  step 4 of Algorithm~\ref{alg:bcfw} can be run in parallel~\cite{stofw} as illustrated in~Fig.~\ref{fig:bcfw}. In this case, $B$ can be selected according to the number of physical processor cores in the control center. 

\textcolor{black}{The only step-size sequence for RB-FW available in the literature is}~\cite{stofw} 
\begin{equation}\label{eq:badstep}
\gamma_t = \frac{2\alpha}{\alpha^2t+2/N_b}, \quad  t =0, 1,\ldots
\end{equation}
where $\alpha:=B/N_b$ is the fraction of updated blocks.
For $\alpha =1$ and $N_b \neq 1$, note that~\eqref{eq:badstep} is different from~\eqref{eq:typicalr}; hence, FW is not generally a special case of the parallel RB-FW in~\cite{stofw}. Interestingly,  Sec.~\ref{sec:bcfw} will introduce a family of step sizes for RB-FW that \textcolor{black}{subsumes the one in~\eqref{eq:typicalr} as a special case}. 

\begin{algorithm}[t]
	\caption{Randomized Block Frank-Wolfe}\label{alg:bcfw}
	\begin{algorithmic}[1]
		\renewcommand{\algorithmicrequire}{\textbf{Input:}}
		\renewcommand{\algorithmicensure}{\textbf{Output:}}
		\State Initialize $t=0$, $\mathbf{x}^0\in \mathcal{X}$
		\While{stopping criterion not met} 
		\State Randomly pick $\mathcal{B}_t\subseteq \mathcal{N}_b$ such that $|\mathcal{B}_t|=B$
		\State \textcolor{black}{Compute $\mathbf{s}_n^t= 
			\underset{\mathbf{s}_n \in \mathcal{X}_n}{\arg\min} \quad \mathbf{s}_n^{\top} \nabla_{\x_n} f(\x^t),~\forall n \in \mathcal{B}_t$}
		\State Update 
		\begin{equation*}
			\quad 	\quad 	\quad	\quad	\mathbf{x}_n^{t+1}=\left\{
			\begin{array}{ll}
				(1-\gamma_t)\mathbf{x}_n^t+\gamma_t\mathbf{s}_n^t,& ~\forall n \in \mathcal{B}_t\\
				\mathbf{x}_n^t ,&~\forall n \in \mathcal{N}_b \setminus \mathcal{B}_t
			\end{array} \right.
		\end{equation*}
		\State $t \leftarrow t+1$
		\EndWhile
	\end{algorithmic}
\end{algorithm}

\begin{figure}[t]
	\centering
	\includegraphics[width=0.5\textwidth]{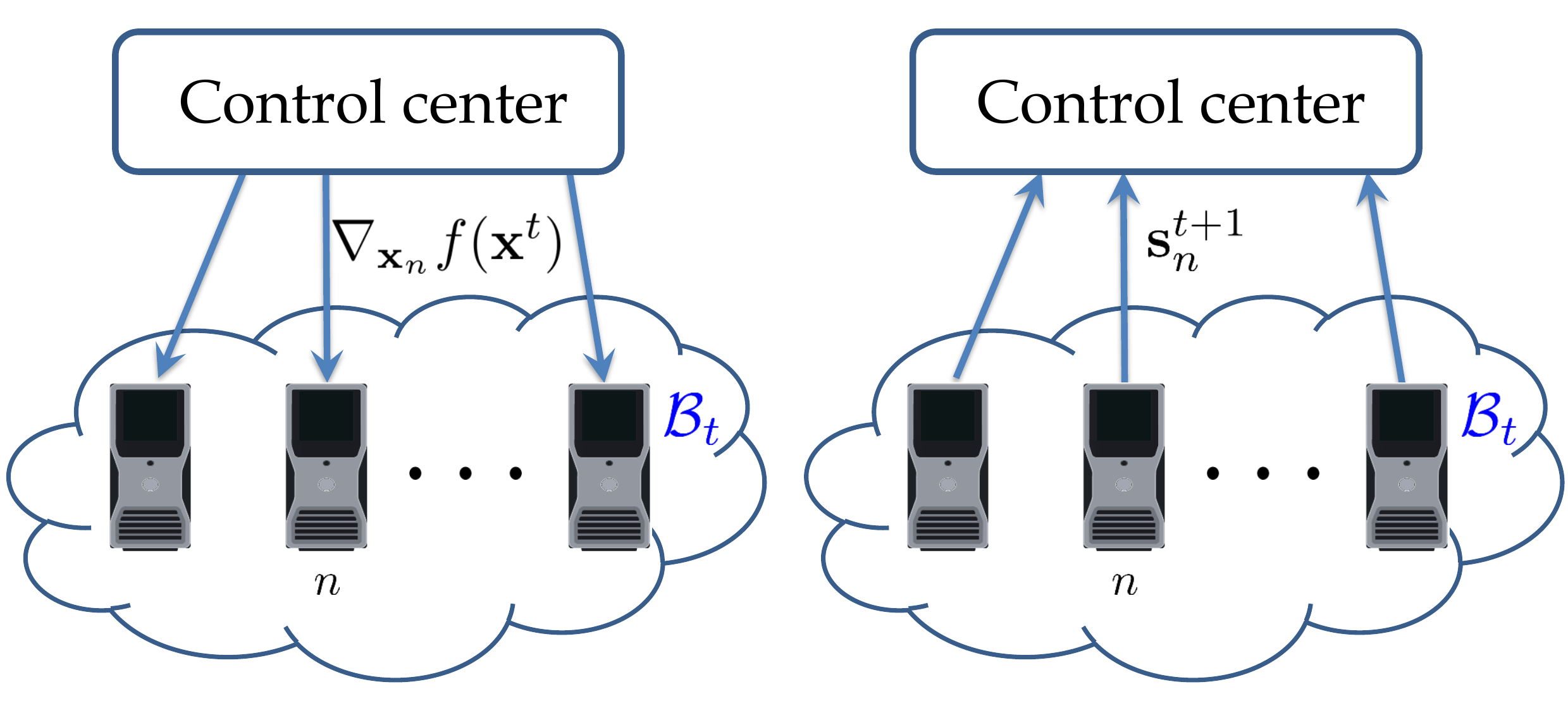} 
	\caption{Parallel implementation for Algorithm~\ref{alg:bcfw} at iteration $t\geq0$. \emph{Left:} The control center sends  gradient $\nabla_{\x_n} f(\x^t)$ to processor $n \in \mathcal{B}_t$. \emph{Right:} The updated $\{\s_n^{t+1}\}_{n\in \mathcal{B}_t}$ are sent to the control center.}\label{fig:bcfw}
\end{figure}


Regarding convergence of FW solvers, two quantities play instrumental roles. The first one is the curvature constant, which  for a differentiable $f(\x)$ over $\mathcal{X}$ is defined as~\cite{clarkson2010coresets},~\cite{jaggi2013revisiting}
\begin{equation}\label{eq:cf}
C_f:=\sup_{{\gamma\in[0, 1]   \atop \x, \s \in \mathcal{X}  }\atop \y := (1-\gamma)\x + \gamma \s } \frac{2}{\gamma^2} \left[f(\y) - f(\x)- \langle \y -\x, \nabla f(\x) \rangle \right].
\end{equation}
$C_f$ is the least upper bound of a scaled difference between $f(\y)$ and its linear approximation around $\x$. Throughout,  $C_f$ is assumed bounded. This property is closely related to the $L$-Lipschitz continuity of  $\nabla f(\x)$  over $\mathcal{X}$, which is defined~as
\begin{equation}\label{eq:lc_ass}
\exists L>0:~~ \|\nabla f(\x)-\nabla f(\s) \| \leq L\|\x-\s \|,~~ \forall \x,~\s \in \mathcal{X}. 
 \end{equation}
If~\eqref{eq:lc_ass} holds, it is easy to check that~\cite[Appendix D]{jaggi2013revisiting} 
\begin{equation}\label{eq:c_f}
C_f \leq LD_{\mathcal{X}}^2
\end{equation}
where $D_{\mathcal{X}}: = \sup_{\x,\s \in \mathcal{X}}\|\x-\s \|$ is the diameter of $\mathcal{X} $, which is finite for $\mathcal{X}$ compact.  Equation~\eqref{eq:c_f} evidences that $C_f$ is bounded whenever  $\nabla f(\x)$ 
is $L$-Lipschitz continuous over~$\mathcal{X}$. 

When it comes to RB-FW, the set curvature for an index set $\mathcal{B} \subseteq \mathcal{N}_b$ is commonly used instead of the constant $C_f$~\cite{stofw}  
\begin{equation}\label{eq:setcurv}
	C_f^{\mathcal{B}}:=\sup_{{{ \gamma \in  [0, 1] \atop \x \in \mathcal{X}} } \atop \{\s_n \in \mathcal{X}_{n}\}_{n\in\mathcal{B}} }  \frac{2}{\gamma^2} \Big(f(\y) - f(\x)- \sum_{n\in\mathcal{B}}\langle \y_n -\x_n, \nabla_{\x_n} f(\x) \rangle \Big)
\end{equation}
where
\begin{equation*}\label{eq:y_n}
	\y_n :=\left\{ \begin{array}{ll}
		(1-\gamma)\x_n + \gamma \s_n ,		&n\in \mathcal{B}\\
		\x_n ,		&n\in \mathcal{N}_b \setminus\mathcal{B}\\
	\end{array}\right.
\end{equation*}
and $\y^\top := [\y_1^\top,\ldots ,\y_{N_b}^\top]$. The \emph{expected set curvature} for a $\mathcal{B}$ selected uniformly at random with $|\mathcal{B}|=B$  can thus be expressed as
\begin{equation}
	\bar C_f^{B}:= \mathbb{E}_{\mathcal{B}}\left[C_f^{\mathcal{B}}\right]= \binom{N_b}{B}^{-1}\sum_{\{\mathcal{B}:~\mathcal{B} \subseteq \mathcal{N}_b, |\mathcal{B}| =B \}} C_f^{\mathcal{B}}
\end{equation}
where  $\binom{N_b}{B}:= N_b!/\big( B!(N_b-B)! \big)$. It is easy to verify that $\bar{C}_f^{B} \leq C_f$ by observing from~\eqref{eq:cf} and~\eqref{eq:setcurv}  that $C_f^{\mathcal{B}} \leq C_f $, $\forall \mathcal{B} \subseteq \mathcal{N}_b$. Note however that $\bar{C}_f^{B} = C_f^{\mathcal{B}}  = C_f$,  when $B=N_b$.

The second quantity of interest is the so-termed \emph{duality gap}  
\begin{equation}\label{eq:gxk0}
g(\x) : = \sup_{\s \in \mathcal{X}}~(\x - \s)^\top \nabla f(\x), \quad \x\in \mathcal{X}
\end{equation}
whose name  stems from Fenchel duality; see~\cite[Appendix D]{lacoste2013block},~\cite[Section 2]{jaggi2013revisiting}. Clearly, for the constrained problem~\eqref{eq:problem}, $\x$ is a stationary point if and only if $g(\x)=0$. In addition,  it holds that $g(\x)\geq 0,~\forall \x\in \mathcal{X}$, since $(\x - \s)^\top \nabla f(\x)=0$ for $\s = \x$. Thus, $g(\x)$ quantifies the distance of $\x$ from a stationary point of $f(\x)$~\cite{lacoste2016convergence}. 


\section{Feasibility-Ensuring Step Sizes for RB-FW}\label{sec:bcfw}

To motivate the need for novel step sizes, this section starts by showing that $\gamma_t$ in~\eqref{eq:badstep} does not guarantee feasibility of the iterates $\{\mathbf{x}^t\}$. It then introduces two families of feasibility-ensuring step size sequences, and proves that the iterates they generate are convergent for convex objectives. Moreover, these families are shown to offer a  gamut of decay rates, thereby allowing for a flexible selection of the most suitable step size for a~given~problem. 

With e.g., $N_b=10^3$ and $B=2$, the step size in~\eqref{eq:badstep} will be $\gamma_t >1,~\forall t< 500$. \textcolor{black}{As a result, step 5 of 
Algorithm~\ref{alg:bcfw} can generate infeasible iterates $\x^t$, which render RB-FW unstable since the gradient of the objective at the resulting $\x^t$ may not even be defined.} For example, consider applying Algorithm~\ref{alg:bcfw} with $N_b=100,~B=10$ and step size as in~\eqref{eq:badstep}  to solve the smooth and convex program
\begin{align}\label{eq:ex}
	\minimize_{ \{x_n\}_{n=1}^{100}}\quad &\sum_{n=1}^{100} (x_n)^2-\log x_n  \\
	\subjectto \quad & 2\leq x_n\leq 3, \quad  n = 1,\ldots, 100\nonumber.
\end{align}

Initializing with $\{x_n^1 =3\}_{n\in \mathcal{N}_b}$, it is easy to verify that  $\{s_n^1 =2\}_{n\in \mathcal{B}_1}$ and $\{x_n^2 =-11/3\}_{n\in \mathcal{B}_1}$, implying that $f(\x^2)$ and $\nabla f(\x^2)$ do not exist. Thus, the parallel RB-FW algorithm in~\cite{stofw}, whose step size is given by~\eqref{eq:badstep}, cannot solve~\eqref{eq:ex}. 

In a nutshell, existing step sizes do not guarantee feasibility of RB-FW iterates. Besides, the decay rates of existing step sizes can not be flexibly adjusted to optimize convergence in a given problem; see Remark~\ref{remark:1}. To fill this gap, convergence analysis of RB-FW will be pursued first for a rich class of step sizes.

\subsection{Convergence of RB-FW for convex programs} \label{sec:cvxBCFW}
For randomized FW, convergence analysis typically focuses on  $f(\x^t)$ and  $g(\x^t)$ in~\eqref{eq:gxk0}~\cite{lacoste2013block, stofw}. 
Let $\x^*$ denote one globally optimal solution of~\eqref{eq:blockproblem}, and define  the \emph{primal sub-optimality} of  $\x^t$ as $h(\x^t):= f(\x^t) - f(\x^*)$.
The next lemma, which quantifies the improvement of $h(\x^t)$ per iteration,  will prove handy in the ensuing convergence~analysis.  
\begin{lemma} \label{lem:declem} 
If  $\{\x^t\}_{t=0,1,\ldots}$ is generated by Algorithm~\ref{alg:bcfw} with an arbitrary predefined step-size sequence $\{ \gamma_t\}_{t=0,1,\ldots} $ satisfying $\gamma_t \in [0,1]~\forall t$, then it holds that
\begin{equation}\label{eq:dec}
\mathbb E\left[ h(\x^{t+1}) \right] \leq\mathbb E\left[ h(\x^t)\right] -\alpha \gamma_t \mathbb E\left[g(\x^t)\right]+{\gamma_t^2 \bar C_f^{B}}/{2}  
\end{equation}
for $t\ge 0$, where the expectation is taken over $\{\mathcal{B_{\tau}}\}_{\tau = 0}^{t} $.
\end{lemma}
A detailed proof can be found in~\cite{lacoste2013block, stofw}; see also part \ref{app:A} of the Appendix for an outline. Note that Lemma~\ref{lem:declem} can be applied regardless of whether  $f(\x)$ is convex or not.





Aiming to upper bound $\mathbb E \left[ h(\mathbf{x}^{t})\right]$, we will consider that  $\{\gamma_t\}_{t=0,1,\ldots}$ satisfy
\begin{subequations}\label{eq:qualcon}
	\begin{align}
     0 < \gamma_t \leq 1, &\quad \forall t\geq 0 \label{eq:qualcona}\\
     \frac{1-\alpha \gamma_{t+1}}{\gamma_{t+1} ^2}  \leq \frac{1}{\gamma_{t} ^2}, &\quad \forall t\geq 0 \label{eq:qualconb}.
	\end{align}
\end{subequations}
It can be easily seen that~\eqref{eq:qualconb} is equivalent to 
\begin{equation*}
\gamma_{t+1} \ge \frac{\gamma_t}{2} \left(\sqrt{\alpha^2\gamma_t^2 +4} -\alpha \gamma_t\right)
\end{equation*}
which implies that \eqref{eq:qualconb} limits how rapidly $\{\gamma_t\}_{t = 0,1, \ldots}$ can decrease. 
Condition~\eqref{eq:qualcon} is very general and subsumes existing step sizes as special cases. 
For example, if $B=N_b$, Algorithm~\ref{alg:bcfw} boils down to Algorithm~\ref{alg:fw}, for which \eqref{eq:typicalr} is typically adopted~\cite{jaggi2013revisiting}. For $\gamma_t$ as in~\eqref{eq:typicalr}, it is clear that \eqref{eq:qualcona} is satisfied, whereas \eqref{eq:qualconb} follows from $(t+1)(t+3)\leq (t+2)^2$. Another example arises if $B=1$, in which case Algorithm~\ref{alg:bcfw} reduces to Algorithm 4 in~\cite{lacoste2013block}. The sequence 
\begin{equation}\label{eq:bcfwgamma}
	\gamma_t = \frac{2N_b}{t+2N_b}
\end{equation}
which was proposed in~\cite{lacoste2013block} for Algorithm~\ref{alg:bcfw}, clearly satisfies \eqref{eq:qualcona}, and also \eqref{eq:qualconb} since it holds that 
$(t+1+2N_b)^2 -  2(t+1+2N_b) \leq (t+2N_b).$ The step size~\eqref{eq:badstep} also satisfies~\eqref{eq:qualconb} since $(\alpha^2 t+\alpha^2+2/N_b)^2 -  2\alpha^2(\alpha^2 t+\alpha^2+2/N_b) \leq (\alpha^2 t+ 2/N_b)^2$, but fails to satisfy~\eqref{eq:qualcona}, which ensures feasible iterates. However, upon observing that $\gamma_t$ in~\eqref{eq:badstep} satisfies $\gamma_t \leq 1$ for $t \geq \tilde{t}:= (2BN_b-2N_b)/B^2$, one deduces that the shifted sequence
\begin{equation}\label{eq:gammatilde}
\tilde{\gamma}_t := \gamma_{t+\tilde{t}} =  \frac{2}{\alpha t +2} 
\end{equation}	
does satisfy~\eqref{eq:qualcona}, and therefore constitutes a feasible alternative to~\eqref{eq:badstep}. Furthermore, it also satisfies \eqref{eq:qualconb} because $(\alpha t+2 +\alpha)(\alpha t+2 -\alpha) \leq (\alpha t  +2)^2$.

To proceed with convergence rate analysis for a broad class of step sizes, an upper bound on $\mathbb E \left[ h(\mathbf{x}^{t})\right]$ for step sizes satisfying   \eqref{eq:qualcon} will be developed.

\begin{theorem}[Primal convergence]
	\label{thm:stocon}
 If  $f(\mathbf{x})$ is convex and  $\{\gamma_t\}_{t=0,1,\ldots}$ satisfies~\eqref{eq:qualcon}, the iterates of Algorithm~\ref{alg:bcfw} satisfy  
\begin{equation}\label{eq:storesult}
\mathbb E \left[ h(\mathbf{x}^{t}) \right]\leq \frac{1-\alpha\gamma_{0}}{\gamma_{0}^2}\gamma_{t-1}^2 h(\mathbf{x}^{0}) + \frac{t\bar{C}_f^{B}}{2}\gamma_{t-1}^2,\quad t\ge 1.
\end{equation}
\end{theorem}
\begin{proof}
Since $f(\x)$ is differentiable, convexity of $f(\x)$ implies that
\begin{equation}\label{eq:cvxg}
f(\x^t) - f(\x^*) \leq (\x^t-\x^*)^{\top}\nabla f(\x^t)
\end{equation}
where $\x^*$ denotes any solution to~\eqref{eq:blockproblem}. Combining~\eqref{eq:gxk0} and~\eqref{eq:cvxg}  yields
\begin{equation} \label{eq:gteq}
g(\x^t) \geq f(\x^t)-f(\x^*) = h(\x^t)\geq 0.
\end{equation}	
Thus, $\mathbb{E}[g(\x^t)] \geq \mathbb{E}[h(\x^t)]$ and  \eqref{eq:dec} can be rewritten as	
\begin{equation}\label{eq:dec1}
\mathbb E\left[ h(\x^{t+1}) \right] \leq (1-\alpha \gamma_t) \mathbb E\left[ h(\x^t)\right]  +{\gamma_t^2 \bar C_f^{B}}/{2} . 
\end{equation}	
Dividing both sides of~\eqref{eq:dec1} by $\gamma_t^2$ gives rise to
\begin{equation}\label{eq:dec2}
\frac{1}{\gamma_t^2}\mathbb E\left[ h(\x^{t+1}) \right] \leq \frac{1-\alpha \gamma_t}{\gamma_t^2} \mathbb{E}\left[ h(\x^t)\right]+ \frac{ \bar{C}_f^{B}}{2}.
\end{equation}
Utilizing successively~\eqref{eq:qualconb} and \eqref{eq:dec2} yields
\begin{align}
\frac{1}{\gamma_t^2}\mathbb{E}[h(\mathbf{x}^{t+1})]
&\leq \frac{1}{\gamma_{t-1}^2}\mathbb{E}[ h(\mathbf{x}^t)]+ \frac{1}{2}\bar{C}_f^{B}\nonumber\\
&\leq \frac{1-\alpha\gamma_{t-1}}{\gamma_{t-1}^2}\mathbb{E}[ h(\mathbf{x}^{t-1})]+ \frac{1}{2}\bar{C}_f^{B} + \frac{1}{2}\bar{C}_f^{B}\nonumber\\
&\leq \ldots \leq \frac{1-\alpha\gamma_{0}}{\gamma_{0}^2} h(\mathbf{x}^{0})+ \frac{t+1}{2}\bar{C}_f^{B} \label{eq:iteration_com}
\end{align}
where the last inequality uses $\mathbb{E}[h(\x^0)] = h(\x^0)$. Therefore, 
\begin{equation*}
\mathbb E \left[ h(\mathbf{x}^{t+1}) \right]\leq \frac{1-\alpha\gamma_{0}}{\gamma_{0}^2}\gamma_t^2 h(\mathbf{x}^{0}) + \frac{t+1}{2}\gamma_t^2 \bar{C}_f^{B}
\end{equation*}
which establishes~\eqref{eq:storesult}.
\end{proof}

Theorem~\ref{thm:stocon} generalizes existing results on the convergence of Algorithm~\ref{alg:bcfw}, which apply only for specific step sizes either assume $B=1$~\cite{lacoste2013block} or $B=N_b$~\cite{jaggi2013revisiting}. Thus, Theorem~\ref{thm:stocon} sheds light on  step size design for arbitrary $B$ by providing  computational guarantees for Algorithm~\ref{alg:bcfw} with any step-size sequence satisfying~\eqref{eq:qualcon}.

Another quantity of interest to characterize the convergence of Algorithm~\ref{alg:bcfw} is  $g(\x^t)$, which can be used to assess how close is $\x^t$ from being a solution~\cite{lacoste2013block},~\cite{stofw} since $g(\x^t) \ge h(\x^t)$; cf.~\eqref{eq:gteq}. 
However, since finding upper bounds on  $ g(\x^t)$ is difficult~\cite{lacoste2013block},~\cite{lacoste2016convergence},~\cite{stofw},  bounds on the minimal expected duality gap until iteration $t$,  defined as~\cite{lacoste2013block},~\cite{stofw}
 \begin{equation}\label{eq:gt}
 g_t : = \min_{k \in  \{0,1, \ldots t\}}  \mathbb E \left[g(\x^k)\right]
 \end{equation}
 are pursued next.

\begin{theorem}[Primal-dual convergence]\label{thm:stodualcon}
Let $\{\gamma_t\}_{t=0,1,\ldots}$ satisfy~\eqref{eq:qualcon} and $\gamma_{t+1} \leq \gamma_t,~\forall t\geq 0 $.
If \textcolor{black}{$f(\mathbf{x})$ is convex} and $\{ \x^t\}_{t=1,2,\ldots}$ is generated by Algorithm~\ref{alg:bcfw}, then for all $K \in \{1,\ldots, t \}$ it holds that
\begin{equation} \label{eq:stodualcon}
{g}_t \leq \frac{\mathbb{E} \left[h(\x^K)\right]}{\alpha(t-K+1)\gamma_t}+ \frac{\bar{C}_f^{B}\gamma_K^2}{2\alpha\gamma_t}\;.
\end{equation}

\end{theorem}

\begin{proof}
	Lemma~\ref{lem:declem} asserts that
	\begin{equation}\label{eq:dualdecrease}
	\alpha \gamma_k \mathbb E[g(\x^k)] \leq  \mathbb E\left[ h(\x^{k}) \right] -\mathbb E\left[ h(\x^{k+1})\right]+{\gamma_k^2 \bar C_f^{B}}/{2}. 
	\end{equation}	
From $g_t\leq \mathbb{E}[g(\x^k)]$ and~\eqref{eq:dualdecrease}, it follows that
\begin{align}
\alpha {g}_t \sum_{k = K}^t \gamma_k \leq& \alpha \sum_{k = K}^t \gamma_k \mathbb{E} [ g(\x^k)] \nonumber\\
\leq & \sum_{k = K}^t  \left(\mathbb E\left[ h(\x^{k}) \right] -\mathbb E\left[ h(\x^{k+1})\right]\right)  +\frac{\bar{C}_f^{B}}{2}\sum_{k = K}^t {\gamma}^2_k \nonumber \\
                                   =  &\mathbb E\left[ h(\x^{K}) \right] -\mathbb E\left[ h(\x^{t+1}) \right ] +\frac{\bar{C}_f^{B}}{2} \sum_{k = {K}}^t {\gamma}^2_k \nonumber \\
                                  \leq & \mathbb{E}\left[ h(\x ^{K})\right] +\frac{\bar{C}_f^{B}}{2} (t-K+1) {\gamma}^2_K \label{eq:deccvxg}
\end{align}
where the last inequality follows from $\mathbb{E}[h(\x^{t+1})] \ge 0$ and $\gamma_{t+1} \leq \gamma_t$. But since 
$\gamma_k \geq \gamma_t,~\forall k \leq t$, one arrives at
\begin{equation}\label{eq:Kg}
 \sum_{k = K}^t \gamma_k \geq (t-K+1)\gamma_t.
\end{equation}
Finally, \eqref{eq:stodualcon} follows after combining \eqref{eq:deccvxg}  with~\eqref{eq:Kg}, and dividing both sides of the resulting inequality by $\alpha (t-K+1)\gamma_t$.
\end{proof}

Theorem~\ref{thm:stodualcon} characterizes the primal-dual convergence of RB-FW for any non-increasing step size satisfying~\eqref{eq:qualcon}. 
\textcolor{black}{
Plugging~\eqref{eq:storesult} into~\eqref{eq:stodualcon} and fixing the step-size sequence yields an upper bound on $g_t$ that can be minimized with respect to $K$ to obtain the convergence rate of $g_t$. This approach will be pursued in Section~\ref{sec:proposed}. }
 

\subsection{Proposed step sizes} \label{sec:proposed}
 This section develops  two classes of step sizes obeying~\eqref{eq:qualcon}  for arbitrary values of $B$. Theorems~\ref{thm:stocon} and~\ref{thm:stodualcon} will then be invoked to derive the resulting convergence rates. To start with, consider the following general   family of diminishing step-size sequences for fixed
 $q\in (0, \alpha]$ and decay rate $\rho \in (0.5, 1]$:
\begin{equation} \label{eq:slowstepsto}
\boxed{
\gamma_t = \frac{2}{q t^\rho +2},\quad \forall t\ge 0.
}
\end{equation}
As will be seen, this family includes, as special cases, the step sizes in~\eqref{eq:typicalr}, \textcolor{black}{\eqref{eq:bcfwgamma}, and \eqref{eq:gammatilde}}.

\begin{lemma}\label{prop:slowstepsto} 
If $\{\gamma_t\}_{t=0,1,\ldots}$ is given by~\eqref{eq:slowstepsto}, it satisfies~\eqref{eq:qualconb}.
\end{lemma}
\begin{proof}
See part~\ref{app:C} of the Appendix.
\end{proof}

Upon noticing that $0< \gamma_t \leq 1$ and $ \gamma_{t+1} \leq  \gamma_t $ for $\{\gamma_t\}_{t=0,1,\ldots} $ in~\eqref{eq:slowstepsto},
 the convergence rate of RB-FW can be derived by appealing to Theorems~\ref{thm:stocon} and \ref{thm:stodualcon} as follows.
\begin{corollary} \label{cor:slowstep}
For convex $f(\mathbf{x})$, the iterates $\{\mathbf{x}^t\}_{t=1,2,\ldots}$ of Algorithm~\ref{alg:bcfw} with step size~\eqref{eq:slowstepsto}  satisfy
\begin{equation}\label{eq:stoprimalcon}
\mathbb E \left[ h(\mathbf{x}^{t}) \right]\leq \frac{4\left(1-\alpha\right) h(\mathbf{x}^{0})}{\left[q(t-1)^\rho +2\right]^2 } + \frac{2t\bar{C}_f^{B}}{\left[q(t-1)^\rho +2\right]^2} 
\end{equation}
 and
\begin{equation} \label{eq:stodual}
{g}_t \leq \frac{(2\rho+1)^{2\rho+1}(qt^\rho +2) }{\alpha q^2(2\rho)^{2\rho}} \cdot \frac{(t+1)\bar{C}_f^{B} + 2(1-\alpha)h(\x^0)}{t^{2\rho+1}}.
\end{equation}
\end{corollary}
\begin{proof}
	See part~\ref{app:D} of the Appendix.
\end{proof}
Corollary~\ref{cor:slowstep} subsumes existing convergence results as special cases. Indeed, when $B=N_b$, one has that $\mathcal{B}_t = \mathcal{N}_b~\forall t$, which implies that $\bar{C}_f^{B}= C_f$, and Algorithm~\ref{alg:bcfw} reduces to the traditional FW solver. By selecting $q = 1$ and $\rho = 1$, the classical step size in~\eqref{eq:typicalr} is retrieved. From Corollary~\ref{cor:slowstep}, the resulting computational bounds  are  
\begin{equation} \label{eq:h}
 h(\mathbf{x}^{t}) \leq  \frac{2tC_f}{(t+1)^2} \leq \frac{2C_f}{t+2}
\end{equation}
 and
\begin{equation} \label{eq:g}
{g}_t \leq \frac{27C_f }{4}\cdot\frac{(t+1)(t+2)}{t^3}.
\end{equation}
The resulting convergence rate of $h(\x^t)$ coincides with the one in \cite[Theorem 1]{jaggi2013revisiting}.  As for $ g_t$, the bound in \eqref{eq:g} is of the same order as that in \cite[Theorem 2]{jaggi2013revisiting}. 
 
In addition, with $B=1$,  $q = 1/N_b$, and $\rho = 1$, the step size~\eqref{eq:bcfwgamma} proposed in~\cite{lacoste2013block} is recovered.  
From Corollary~\ref{cor:slowstep}, the primal computational bound is 
\begin{align} \label{eq:h1blcok}
 \mathbb{E}[h(\mathbf{x}^{t})] &\leq \frac{4(N_b^2 -N_b)h(\x^0)}{(t-1 +2N_b)^2} + \frac{2tN_b^2C_f^1}{(t-1 +2N_b)^2} \nonumber \\
 &\leq  \frac{4(N_b^2 -N_b)h(\x^0)}{(t-1 +2N_b)^2} +  \frac{2N_b^2 C_f^1}{t+4N_b-2}
\end{align}
where the last inequality follows from
\begin{equation*}
\frac{t}{t+2N_b-1} \leq \frac{t+2N_b-1}{t+4N_b -2}.
\end{equation*} 
Meanwhile, $g_t$ is bounded by
\begin{equation} \label{eq:g1block}
{g}_t \leq \frac{27N_b (t+2N_b) }{4t^3}\left[(t+1)N_b C_f^1 + 2(N_b-1)h(\x^0)\right].
\end{equation}   
Notably, the bound in~\eqref{eq:h1blcok} is \emph{tighter} than the one reported in~\cite[Theorem 2]{lacoste2013block}, while the  bound on $g_t$ in~\eqref{eq:g1block} is of the same order as that  in~\cite[Theorem 2]{lacoste2013block}. 

Finally, note that  Corollary~\ref{cor:slowstep} also characterizes convergence for the step size $\tilde{\gamma}_t$ in~\eqref{eq:gammatilde}, since $\tilde{\gamma}_t$ is recovered from~\eqref{eq:slowstepsto} upon setting $q = \alpha$ and $\rho =1$.

The decreasing rates of the bounds in Theorems~\ref{thm:stocon} and~\ref{thm:stodualcon} are determined by the decay rates of the step size sequence. The faster $\gamma_t$ diminishes, the more rapidly the upper bound in Theorem~\ref{thm:stocon} vanishes.  However, the sequence in~\eqref{eq:slowstepsto} decreases at most as fast as ${2}/{(\alpha t+2) }$. To improve the bound in Theorem~\ref{thm:stocon}, a more rapidly vanishing sequence is proposed next. Specifically, consider the sequence 
\begin{equation} \label{eq:recursesto}
\boxed{
\gamma_0 = 1,~\text{and} ~\gamma_{t+1} = \frac{\sqrt{\alpha^2 \gamma_t^4+4\gamma_t^2} - \alpha\gamma_t^2}{2}, ~~\forall t\ge 0.}
\end{equation}
It is then possible to establish the following.
\begin{lemma} [Recursive step size]
\label{prop:faststepsto}
If $\{\gamma_t\}_{t=0,1,\ldots}$ is chosen as in~\eqref{eq:recursesto}, it then holds that 
\begin{subequations}\label{eq:storecstep}
	\begin{align}
		\frac{1}{\alpha t +1} \leq \gamma_{t} \leq \frac{2}{\alpha t+2}, &\quad \forall t\geq 0 \label{eq:storecstepa}\\
			\gamma_{t+1} \leq \gamma_t, &\quad \forall t\geq 0 \label{eq:storecstepb}.
	\end{align}
\end{subequations}
\end{lemma} 
\begin{proof}
	See part~\ref{app:E} of the Appendix.
\end{proof}
The upper bound in~\eqref{eq:storecstepa}  confirms that the step size  in~\eqref{eq:recursesto} vanishes at least as fast as ${2}/{(\alpha t+2)}$. To check whether~\eqref{eq:recursesto} meets~\eqref{eq:qualcon}, note that \eqref{eq:qualcona} follows from \eqref{eq:storecstepa}, whereas~\eqref{eq:recursesto} implies that~\eqref{eq:qualconb} holds with equality.  
Because~\eqref{eq:recursesto} satisfies \eqref{eq:qualcon} and \eqref{eq:storecstepb},  the following computational bounds for \eqref{eq:recursesto} can be derived by plugging  \eqref{eq:storecstepa} into Theorems~\ref{thm:stocon}~and~\ref{thm:stodualcon}.

\begin{corollary} \label{cor:recf}
For convex $f(\mathbf{x})$, the iterates $\{\mathbf{x}^t\}_{t=1,2,\ldots}$ of Algorithm~\ref{alg:bcfw} with step size as in~\eqref{eq:recursesto},  satisfy
\begin{equation}\label{eq:fastfcon}
\mathbb E \left[ h(\mathbf{x}^{t}) \right]\leq \frac{4(1-\alpha) h(\mathbf{x}^{0})}{(\alpha t +2-\alpha)^2} + \frac{2t\bar{C}_f^{B} }{(\alpha t +2-\alpha)^2 }
\end{equation}
 and
\begin{equation} \label{eq:fastgcon}
{g}_t \leq \frac{27(\alpha t +1)}{2\alpha^3 t^3}\left[{(t+1)\bar{C}_f^{B} + 2(1-\alpha)h(\x^0)}  \right].
\end{equation}
\end{corollary}
\begin{proof}
	See part~\ref{app:F} of the Appendix.
\end{proof}

To recap, this section put forth two families of step sizes for Algorithm~\ref{alg:bcfw} with arbitrary $B$, namely \eqref{eq:slowstepsto} and \eqref{eq:recursesto}. Corollaries~\ref{cor:slowstep} and~\ref{cor:recf} establish convergence of Algorithm~\ref{alg:bcfw} for these step sizes, which also guarantee feasibility of the iterates since they satisfy~\eqref{eq:qualcona}. When $\{\gamma_t\}_{t=0,1,\ldots}$ is given by~\eqref{eq:slowstepsto} with $q=\alpha$ and $\rho =1$ or when it is defined as in ~\eqref{eq:recursesto},  the convergence rates of Algorithm~\ref{alg:bcfw} are in the order of $\mathcal{O}\left({1}/{t}\right)$, thus matching those of the traditional FW algorithm, yet the computational cost of the former is potentially much lower than that of the latter. 

\begin{remark}\label{remark:line-search}
The step size of RB-FW can also be chosen through line search, which prescribes 
\begin{equation} \label{eq:linesearchBCFW}
\gamma_t  = \arg \min_{0\leq \gamma \leq 1} f\left((1-\gamma)\mathbf{x}^t+\gamma \hat{\mathbf{s}}^t\right)
\end{equation}
with ${\hat{\s}^t}^\top := [{\hat{\s}_1^t}^\top , \ldots,   {\hat{\s}_{N_b}^t}^\top ]$ and
\begin{equation*}\label{eq:hats_n}
\hat{\s}_n^t :=\left\{ \begin{array}{ll}
\s_n^t ,		&n\in \mathcal{B}_t\\
\x_n^t ,		&n\in \mathcal{N}_b \setminus\mathcal{B}_t.\\
\end{array}\right. 
\end{equation*}
Let  $\{{\check{\x}^{t}}\}_{t=0,1, \ldots}$ be the iterates generated by Algorithm~\ref{alg:bcfw} with $\gamma_t$ given by~\eqref{eq:linesearchBCFW}.
By~\eqref{eq:dec} and~\eqref{eq:linesearchBCFW}, it holds that
\begin{equation}\label{eq:line}
\mathbb E\left[ h(\check{\x}^{t+1}) \right]\leq\mathbb E\left[ h(\check{\x}^t)\right] -\alpha \gamma_t \mathbb E\left[g(\check{\x}^t)\right]+{\gamma_t^2 \bar C_f^{B}}/{2} . 
\end{equation}
for any  predefined step-size sequence $\{\gamma_t \in [0, 1]\}$~\cite{lacoste2013block,stofw}. Particularly,  \eqref{eq:line} holds for $\{\gamma_t :=  2/ (\alpha t +2)\}_{t=0, 1,\ldots}$. It can then be shown that $\{{\check{\x}^{t}}\}_{t=0,1, \ldots}$ satisfy for $t\geq 1$
\begin{equation*}\label{eq:linesearchh}
\mathbb E \left[ h(\check{\mathbf{x}}^{t}) \right]\leq \frac{4(1-\alpha) h(\mathbf{x}^{0})}{(\alpha t +2-\alpha)^2} + \frac{2t\bar{C}_f^{B} }{(\alpha t +2-\alpha)^2 } 
\end{equation*}
and
\begin{equation*} \label{eq:linesearchg}
\check{g}_t \leq \frac{27(\alpha t +2)}{4\alpha^3}\cdot \frac{(t+1)\bar{C}_f^{B} + 2(1-\alpha)h(\x^0)}{t^3}
\end{equation*}
where $\check{g}_t: = \min_{k \in  \{0,1, \ldots t\}}  \mathbb E \left[g(\check{\x}^k)\right]$. The proof follows the steps of the one for Corollary~\ref{cor:slowstep}.
The convergence rate of line-search-based Algorithm~\ref{alg:bcfw} therefore remains in the order of $\mathcal{O}(1/t)$. Note however that extra computational cost is incurred for finding $\gamma_t$ via~\eqref{eq:linesearchBCFW}.

\end{remark}
\color{black}
\begin{remark}\label{remark:1}
At this point, it is worth discussing the choice of the step size leading to the fastest convergence in a given problem. Even though the  bounds in this section suggest that the more rapid the decrease of the step sizes, the quicker the decrease of $h(\x^t)$, this is not always the case in practice. This is because step sizes with large decay rates become small  after the first few iterations, and small  step sizes lead to slow changes in $h(\x^t)$. Conversely,  small decay rates tend to yield rapidly decreasing $h(\x^t)$ in the first few iterations since the step sizes remain relatively large. 
Hence, it is difficult to provide universal guidelines since rapidly or slowly diminishing step sizes may be preferred depending on the specific optimization problem at hand.
For example, if optimal solutions lie in the interior of the feasible set, rapidly diminishing step sizes can help reduce oscillations around optimal solutions, thus improving the overall convergence rates. On the other hand, if $f(\x)$ is monotone on $\mathcal{X}$, the solution lies on the boundary, which means that no oscillatory behavior is produced and, hence, slowly diminishing step sizes will be preferable. 
\end{remark}





 


\section{RB-FW for Nonconvex Programs}\label{sec:bcfwnoncvx} 
The objective function of~\eqref{eq:blockproblem} is nonconvex in certain applications, such as constrained multilinear decomposition~\cite{papalexakis2013k} and power system state estimation~\cite{SPM2013,globalsip2016wzgs}. Yet, convergence of RB-FW has never been investigated for this case. The rest of this section fills this gap by analyzing the convergence rate of RB-FW in problems involving a nonconvex objective. Similar to Sec.~\ref{sec:bcfw}, computational bounds are first derived for a wide class of step sizes, and are subsequently tailored for $\gamma_t$ as in~\eqref{eq:slowstepsto} as well as for  exact line search. 

Recall that Section~\ref{sec:prefw} introduced $g(\x)$ as a  non-stationarity measure of point $\x$ with respect to $f(\x)$. In the sequel, RB-FW with be analyzed in terms of upper bounds on $g_t$ [cf.~\eqref{eq:gt}]. 
\begin{theorem}\label{thm:dualgap}
	If $\{\gamma_t\}_{t=0,1,\ldots}$ satisfy $0 \leq \gamma_t \leq 1 ~\forall t$,  it holds for the iterates $\{\x^t\}_{t=0,1,\ldots}$ of Algorithm~\ref{alg:bcfw} that 
	\begin{equation} \label{eq:ncvxg}
		{g}_t \leq  \frac{h(\x ^0)}{\alpha\sum_{k = 0}^t \gamma_k } +\frac{\bar C_f^{B} \sum_{k = 0}^t {\gamma}^2_k}{2\alpha \sum_{k = 0}^t \gamma_k} ,\quad t\ge 0.  
	\end{equation}
\end{theorem}
\begin{proof}
	Using $0\leq {g}_t \leq \mathbb{E}[ g(\x^k)] $  and \eqref{eq:dualdecrease}, we deduce that 
	\begin{align*}
		\alpha {g}_t \sum_{k = 0}^t \gamma_k \leq & \alpha \sum_{k = 0}^t \gamma_k \mathbb{E}[g(\x^k)]\nonumber\\
		\leq & \sum_{k = 0}^t \Big(\mathbb E\big[ h(\x^{k}) \big] -\mathbb E\big[ h(\x^{k+1})\big] \Big) +({\bar{C}_f^{B}}/{2})\sum_{k = 0}^t {\gamma}^2_k \nonumber \\
		=  & \mathbb E\left[ h(\x^{0}) \right] -\mathbb E\left[ h(\x^{t+1})\right] +({\bar{C}_f^{B}}/{2})\sum_{k = 0}^t {\gamma}^2_k \nonumber \\
		\leq &  h(\x ^0) +({\bar{C}_f^{B}}/{2}) \sum_{k = 0}^t {\gamma}^2_k 
	\end{align*}
	where the last inequality follows from $\mathbb E\left[ h(\x^{t+1})\right] \ge 0$. Dividing both sides by $\alpha\sum_{k = 0}^t \gamma_k$ leads to \eqref{eq:ncvxg}.
\end{proof} 

Clearly, Theorem~\ref{thm:dualgap} affirms that $\lim_{t\rightarrow \infty}{g}_t =0$ if the step-size sequence $\{\gamma_t\}_{t=0,1,\ldots}$ satisfies
\begin{equation*}
	\lim_{t\rightarrow \infty}\sum_{k=0}^{t} \gamma_k = \infty, ~~{\rm and}~~ \lim_{t\rightarrow \infty}\sum_{k=0}^{t} \gamma_k^2 = S
\end{equation*}
for some finite $S>0$.
In other words, if $\{\gamma_t\}_{t=0,1,\ldots}$ is not summable and $\{\gamma_t^2\}_{t=0,1,\ldots}$ is summable, then either $\x_t$ is a stationary point for some $t$, or, a subsequence of $\{\x_t\}_{t=0,1,\ldots}$ converges to a stationary point. 

For any given step size, the convergence rates of RB-FW can be derived through \eqref{eq:ncvxg}. To start with, consider $\{\gamma_t\}_{t=0,1,\ldots}$ in~\eqref{eq:slowstepsto} with $q= \alpha,~\rho=1$; that is, $\gamma_t = {2}/{(\alpha t +2)}$, and note~that
\begin{subequations}\label{eq:ncvx}
	\begin{align}
&	\sum_{k=0 }^{t} \frac{2}{\alpha k+2} \geq \int_{x=0}^{t} \frac{2}{\alpha x+2} dx  = \frac{2}{\alpha}\log\Big (\frac{\alpha t+2}{2} \Big) \label{eq:ncvxub} \\
&	\sum_{k=0 }^{t} \frac{4}{(\alpha k+2)^2} \leq \int_{x=-1}^{t} \frac{4}{(\alpha x+2)^2} dx = \frac{4}{\alpha}\Big(\frac{1}{2-\alpha} - \frac{1}{\alpha t+2}\Big) \label{eq:ncvxlb}.
    \end{align}
\end{subequations}    
By substituting \eqref{eq:ncvx} into Theorem~\ref{thm:dualgap}, it follows that Algorithm~\ref{alg:bcfw} attains a stationary point of a nonconvex program at rate $\mathcal{O} (1/\log t)$.   

This rather slow rate can be substantially improved upon adopting  exact line search for RB-FW.
\begin{theorem}\label{prop:optdualgap}
	If $\{\gamma_t\}_{t=0,1,\ldots}$ is chosen as in~\eqref{eq:linesearchBCFW}, it holds for the iterates $\{\x^t\}_{t=0,1,\ldots}$ of Algorithm~\ref{alg:bcfw} that 
	\textcolor{black}{
	\begin{equation} \label{eq:noncvx}
		{g}_t\leq \frac{\max \left\{2h(\x^0),\bar{C}_f^{B}\right \}}{\alpha\sqrt{t+1}},\quad t\ge 0.
	\end{equation}}
\end{theorem}
\begin{proof}
	The right-hand side of \eqref{eq:dec} is minimized for
	\begin{align}
		\hat{\gamma}_k &= \arg\min_{\gamma \in [0,1]}  \mathbb E\left[ h(\x^{k}) \right] -\alpha \gamma \mathbb E\left[g(\x^k)\right]+{\gamma^2 \bar{C}_f^{B}}/{2} \nonumber\\ 
		&= \min\big\{ 1, {\alpha\mathbb E\left[g(\x^k)\right]}/{\bar{C}_f^{B}} \big\}.  \label{eq:betterstep}
	\end{align}
	
Thus, if $\mathbb E\left[ g(\x^{k}) \right] \geq {\bar{C}_f^{B}}/{\alpha}$, then $\hat{\gamma}_k=1$ and \eqref{eq:dec} becomes 
	\begin{align}
		\mathbb E\left[ h(\x^{k+1}) \right] & \leq \mathbb E\left[ h(\x^{k}) \right] -\alpha \mathbb E\left[ g(\x^{k}) \right] +\bar{C}_f^{B}/{2} \nonumber \\
		&\leq  \mathbb E\left[ h(\x^{k}) \right] -\alpha \mathbb E\left[ g(\x^{k}) \right] /2 \label{eq: R1gammahat1}
	\end{align}
	where the second inequality follows from $  {\bar{C}_f^{B}} \leq \alpha \mathbb E\left[ g(\x^{k}) \right] $.	
	Similarly, if $\mathbb E\left[ g(\x^{k}) \right] < {\bar{C}_f^{B}}/{\alpha}$, then $\hat{\gamma}_k={\alpha\mathbb E\left[g(\x^k)\right]}/{\bar{C}_f^{B}}$ and \eqref{eq:dec} becomes 
	\begin{equation} \label{eq: R1gammahat2}
	 \mathbb E\left[ h(\x^{k+1}) \right] \leq  \mathbb E\left[ h(\x^{k}) \right] -{\alpha^2  \mathbb E\left[g(\x^k)\right]^2}/2{\bar{C}_f^{B}}.
	\end{equation}    
    Combining both cases, \eqref{eq: R1gammahat1} and~\eqref{eq: R1gammahat2} establish that  
	\begin{equation}
	\mathbb E\left[ h(\x^{k+1}) \right] \leq  \mathbb E\left[ h(\x^{k}) \right] - \min \left\{ \frac{\alpha \mathbb E\left[ g(\x^{k}) \right] }{2}, \frac{\alpha^2  \mathbb{E}^2\left[g(\x^k)\right]}{2\bar{C}_f^{B}}\right\}. \label{eq:ncvxopt}
	\end{equation}
	 When $\gamma_k$ is given by~\eqref{eq:linesearchBCFW} with $t=k$, $h(\x^{k+1})$ is not greater than 
	when $\gamma_k = \hat{\gamma}_k$. Therefore,~\eqref{eq:ncvxopt} still holds in the former case. Thus, for $\{\gamma_k\}_{k=0,1,\ldots}$ as in~\eqref{eq:linesearchBCFW}, it follows that 
		\begin{equation}
\min \left\{ \frac{\alpha \mathbb E\left[ g(\x^{k}) \right] }{2}, \frac{\alpha^2  \mathbb E\left[g(\x^k)\right]^2}{2\bar{C}_f^{B}}\right\} \leq 	  \mathbb E\left[ h(\x^{k}) \right] -\mathbb E\left[ h(\x^{k+1}) \right] . \label{eq:ncvxopteq}
	\end{equation}
	Summing \eqref{eq:ncvxopteq} from $k = 0$ to $t$ yields
		\begin{equation}
		 (t+1)\min \left\{ \frac{\alpha g_t }{2}, \frac{\alpha^2  g_t^2}{2\bar{C}_f^{B}}\right\} \leq   h(\x^{0})- \mathbb E\left[ h(\x^{t+1}) \right]  \label{eq:ncvxsum}.
		\end{equation}
Therefore, 	
\begin{equation}\label{eq:boundgt}
	g_t \leq \max \left\{\frac{2h(\x^{0}) }{\alpha (t+1)}, \frac{\sqrt{2\bar{C}_f^{B} h(\x^0)}}{\alpha\sqrt{t+1}}  \right\}.
\end{equation}
Since $t+1 \geq \sqrt{t+1}$	and $\sqrt{2\bar{C}_f^{B} h(\x^0)} \leq  \max\{2h(\x^0), \bar{C}_f^{B}\}$, \eqref{eq:noncvx} holds.
\end{proof}

\color{black}
Theorem~\ref{prop:optdualgap} generalizes the recent result in~\cite{lacoste2016convergence}, which only applies to the classical FW method. 
The improved bound in~\eqref{eq:noncvx} is attained at the price of performing exact line search, which requires the solution to a potentially nonconvex univariate optimization subproblem~\eqref{eq:linesearchBCFW}. It is worth mentioning that an optimal solution to this subproblem can be readily found in a number of cases. For example, if  $f((1-\gamma)\x^t + \gamma \s^t )$ is quadratic in $\gamma$, then $\gamma_t$ can be readily found by evaluating this function at three~points.

All in all, the main contribution here is a convergence rate analysis of RB-FW for minimizing~\eqref{eq:blockproblem} with nonconvex $f(\x)$. Interestingly, when RB-FW relies on step sizes obtained through line search, a stationary point is reached with rate~$\mathcal{O}{(1/\sqrt{t})}$.

\section{Generalized Step Sizes for FW}\label{sec:fw}

The availability of satisfactory step sizes for FW is rather limited. Indeed, besides line search, convergence rate of FW has only been established for $\gamma_t=\frac{2}{t+2}$~\cite{jaggi2013revisiting}, and $\gamma_t=\frac{1}{t+1}$~\cite{freund2016new}. This limits the user's control on convergence of FW iterates; cf. Remark~\ref{remark:1}. To alleviate this limitation, this section examines the usage of step sizes in~\eqref{eq:slowstepsto} and~\eqref{eq:recursesto} in the classical FW solver, namely Algorithm~\ref{alg:fw}. Since the latter can be viewed as a special case of Algorithm~\ref{alg:bcfw} with $B = N_b$, Corollaries~\ref{cor:slowstep} and~\ref{cor:recf} can be leveraged to derive the convergence rates of FW for convex programs with the novel step sizes. Specifically, the ensuing computational bounds hold.

\begin{corollary} \label{prop:fgconv}
If \textcolor{black}{$f(\x)$ is convex} and the step size sequence $\{\gamma_t\}_{t=0,1,\ldots}$ is chosen as in~\eqref{eq:slowstepsto} with $\alpha =1$,  $q\in (0, 1]$ and $ \rho \in (0.5, 1]$, then the successive iterates $\{\mathbf{x}^t\}_{t=1, 2, \ldots}$ of Algorithm~\ref{alg:fw}  satisfy for $t \geq 1$
	\begin{equation}\label{eq:fcon}
	h(\mathbf{x}^{t}) \leq  \frac{2tC_f}{\left[q(t-1)^\rho +2\right]^2} 
	\end{equation}
	and
	\begin{equation} \label{eq:gcon}
	{g}_t \leq \frac{(2\rho+1)^{2\rho+1}(qt^\rho +2) }{ q^2(2\rho)^{2\rho}}\frac{(t+1)C_f }{t^{2\rho+1}} \;.
	\end{equation}
\end{corollary}
\begin{proof}
	This is a special case of Corollary \ref{cor:slowstep} for $\alpha = 1$.
\end{proof}

\begin{corollary} \label{prop:recconv}
	If \textcolor{black}{$f(\x)$ is convex} and the step-size sequence $\{\gamma_t\}_{t=0,1,\ldots}$ is chosen as in~\eqref{eq:recursesto} with $\alpha = 1$,
	then the successive iterates $\{\mathbf{x}^t\}_{t=1, 2, \ldots}$ of Algorithm~\ref{alg:fw}  satisfy for $t \geq 1$
	\textcolor{black}{
	\begin{equation}\label{eq:primalcon}
	h(\mathbf{x}^{t}) \leq  \frac{2C_f}{t+2} 
	\end{equation}
	and
	\begin{equation} \label{eq:pdconv}
	{g}_t \leq \frac{27 C_f}{2}\left(\frac{1}{t}+\frac{2}{t^2}+ \frac{1}{t^3}  \right).
	\end{equation}}
\end{corollary}
\begin{proof}
	Corollary \ref{prop:recconv} follows directly from Corollary~\ref{cor:recf}.
\end{proof}

Corollaries \ref{prop:fgconv} and \ref{prop:recconv} establish convergence rates in terms of both $h(\x^t)$ and $g_t$ for the classical FW method with step sizes of different decay rates.
For a given problem, the most suitable step size can be selected following the guidelines in Remark~\ref{remark:1}. 
Interestingly,
comparing Corollaries~\ref{prop:fgconv} and \ref{prop:recconv} with  Corollaries~\ref{cor:slowstep} and~\ref{cor:recf} reveals that the initial optimality gap $h(\x^0)$ no longer affects the bounds for FW.

\section{Applications}\label{sec:app}  
Two applications where RB-FW exhibits  significant computational advantages over existing alternatives  will be delineated in this section.
\subsection{Coordination of EV charging}\label{sec:ev}
The convex setup of optimal schedules for EV charging in~\cite{liang16frank} is briefly reviewed next.
Suppose that a load aggregator coordinates the charging of $N$ EVs over the $T$ consecutive time slots  $\mathcal{T}:=\{1,\ldots, T\}$ of length $\Delta_\tau$. Let $\mathcal{T}_n\subseteq \mathcal{T}$ denote the time slots in which vehicle $n$ is connected to the power grid,  and let ${p}_n(\tau)$ be the charging rate of EV $n$ at time $\tau$ to be scheduled by the load aggregator. If $\bar{p}_n$ is the charging rate limitation imposed by the battery of vehicle $n$, then ${p}_n(\tau)$ should lie in the interval $[0, \bar{p}_n(\tau)]$
with
\begin{equation*}
	\bar{p}_n(\tau):=\left\{
	\begin{array}{ll}
		\bar{p}_n,&\tau\in \mathcal{T}_n,\\
		0 ,&\text{otherwise}.
	\end{array} \right.
\end{equation*}
  The charging profile for vehicle $n$, denoted by~$\mathbf{p}_n^{\top}:=[p_n(1), \cdots , p_n(T)]$, should therefore belong to the convex and compact set
\begin{equation*}
	\mathcal{P}_n:=\left\{\mathbf{p}_n: \Delta_\tau \mathbf{p}_n^\top \mathbf{1}=R_n,~0 \leq p_n(\tau) \leq \bar{p}_n(\tau),~\forall \tau\in\mathcal{T}\right\}
\end{equation*}
where  $R_n$  represents the total energy needed by EV $n$. 

Given $\{R_n\}_{n=1}^N,~\{\bar{p}_n\}_{n=1}^N$, and $\{\mathcal{T}_n\}_{n=1}^N$, the problem solved by the aggregator is to find the charging profiles minimizing its electricity cost~\cite{liang16frank}; that is, 
\begin{align}\label{eq:EVS2}
	\mathbf{p}^*\in \arg\min_{\mathbf{p}} ~&~  f(\mathbf{p})\\
	\subjectto~&~ \mathbf{p}_n\in{\mathcal{P}}_n,~\forall~n\in\mathcal{N} \nonumber
\end{align}
where $\mathbf{p}^\top:=[\mathbf{p}_1^{\top},\cdots,\mathbf{p}_{N}^{\top}]$ and $\mathcal{N}:=\{1, \ldots, N\}$.  With $\{D(\tau)\}_{\tau=1}^T$ denoting additional known loads, the total cost $f(\p)$ is
\begin{equation}\label{eq:cost}
	f({\mathbf{p}})=\sum_{\tau=1}^{T} \Big(D(\tau)+\sum_{n=1}^{N}p_n(\tau)\Big)^2\;.
\end{equation}
Note that $f(\p)$ is convex but \emph{not strongly convex} in $\p$. The feasible set for \eqref{eq:EVS2} is the Cartesian product $\mathcal{P}:=\mathcal{P}_1\times\ldots \times \mathcal{P}_{N}$, which is convex and compact. Thus, problem \eqref{eq:EVS2} is convex and of the form~\eqref{eq:blockproblem}. 

Assuming that the aggregator can only afford updating the charging profiles of $B$ out of the $N$ vehicles in parallel due to a limited number of processors, 
the ensuing $B$ linear subproblems arise when solving~\eqref{eq:EVS2} via Algorithm \ref{alg:bcfw}: 
\begin{align}\label{eq:lp3}
\mathbf{s}_n^t\in\arg\min_{\mathbf{s}_n\in \mathcal{P}_n}\langle\s_n, \c^t \rangle,  \quad   n\in \mathcal{B}_t
\end{align}
where $|\mathcal{B}_t|=B$ and  $\c^t := \nabla_{\mathbf{p}_n} f(\p^t)$. The latter does not depend on $n$ since the gradient $\nabla_{\mathbf{p}_n} f(\p^t)$ is identical across the $N$ vehicles. Its $\tau$-th entry is given by 
\begin{equation}\label{eq:c}
c^t(\tau):=2\big(D(\tau)+\sum_{n=1}^N p_n^t(\tau)\big).
\end{equation}

The subproblem \eqref{eq:lp3} can be solved in closed form~\cite{BoVa04}. To find a solution, sort the entries of $\mathbf{c}^t$ in non-decreasing order by finding $\{\tau_i^t\}_{i=1}^{T}$ such that $c^t(\tau_1^t)\leq c^t(\tau_2^t)\leq \ldots \leq c^t(\tau_T^t)$. Subsequently, one needs to find the index $\bar{\tau}_n^t \geq 1$ for which
\begin{equation}\label{eq:opts0}
\sum_{i=1}^{\bar{\tau}_n^t-1}\bar{p}_n(\tau_i^t)\leq R_n~\text{and}~\sum_{i=1}^{\bar{\tau}_n^t}\bar{p}_n(\tau_i^t)> R_n.
\end{equation}
Finally, the entries of the minimizer $\mathbf{s}_n^t$ are found as
\begin{equation}\label{eq:opts}
s_n^{t}(\tau_i^t)=\left\{
\begin{array}{ll}
\bar{p}_n(\tau_i^t),&i=1,\ldots,\bar{\tau}_n^t-1\\
R_n-\sum_{j=1}^{\bar{\tau}_n^t}\bar{p}_n(\tau_j^t),&i=\bar{\tau}_n^t\\
0,&i=\bar{\tau}_n^t+1,\ldots,T.
\end{array}\right.
\end{equation} The  computational advantage of RB-FW  for solving~\eqref{eq:EVS2} stems from the fact that the solution to the  subproblems~\eqref{eq:lp3} can be  obtained efficiently via \eqref{eq:opts} \textcolor{black}{upon receiving the $\mathbf{c}_t$ entry order, whereas competing alternatives require projections onto $\{\mathcal{P}_n\}_{n\in \mathcal{B}_t}$ per iteration~\cite{liang16scalable}}. Our RB-FW-based charging scheme is summarized in Algorithm~\ref{alg:docsev}.
\begin{algorithm}[t]
\caption{EV charging coordination solver}\label{alg:docsev}
\begin{algorithmic}[1]
\renewcommand{\algorithmicrequire}{\textbf{Input:}}
\renewcommand{\algorithmicensure}{\textbf{Output:}} 
\Require  $\{R_n\}_{n=1}^N,~\{\bar{p}_n\}_{n=1}^N$,  $\{\mathcal{T}_n\}_{n=1}^N$, and  $B$
\State Initialize $\{\p_n^0\}$ and $t=0$
\While{stopping\_criterion not met} 
\State Randomly pick $\mathcal{B}_t\subseteq \mathcal{N}$ such that $|\mathcal{B}_t|= B$
\State Evaluate $\c^t$ via~\eqref{eq:c} \textcolor{black}{and broadcast $\c^t$ entry order} 
\State Calculate $\{\s_n^t\}_{n \in \mathcal{B}_t}$ via \eqref{eq:opts0} and \eqref{eq:opts} 
\State Update $\{\p_n^{t+1}\}_{n\in \mathcal{N}}$ via 
		\begin{equation*}
		\quad 	\quad 	\quad	\quad	\p_n^{t+1}=\left\{
		\begin{array}{ll}
		 (1-\gamma_t)\p_n^{t}+\gamma_t\s_n^{t},& ~\forall n \in \mathcal{B}_t\\
		\mathbf{p}_n^t ,&~\forall n \in \mathcal{N} \setminus \mathcal{B}_t
		\end{array} \right.
		\end{equation*}
\State $t \leftarrow t+1$
\EndWhile
\end{algorithmic}
\end{algorithm} 

\subsection{Structural SVMs}\label{subsec:mkr}
The term \emph{structured prediction} comprises a family of machine learning problems, where the output to the \emph{predictors} have variable sizes~\cite{bakir2007predicting}. An example is the optical character recognition (OCR) task, where one is given a vector $z\in \mathbb{R}^P$ containing the $P$-pixel image of an $M$-letter word. The goal is to produce a vector $\y \in \{1,\ldots, 26\}^M$, whose $m$-th entry indicates which of the 26 letters of the alphabet corresponds to the $m$-th character in that word. This problem is challenging because the output $\y$ may take $26^M$ values, and also the same predictor must work for different values of $M$.

 
Structural SVMs have been widely adopted to carry out the aforementioned structured prediction tasks~\cite{taskar2003max},~\cite{tsochantaridis2005large}. Upon defining the application-dependent feature map $\bm{\phi}$~\cite{tsochantaridis2005large} that encodes  the relevant information for the input-output pair $(\z,\y)$ in the $d$-dimensional vector $\bm{\phi}(\z,\y)$, a vector $\w$  is learned so that  $\langle \w , \bm{\phi}(\z, \y) \rangle$ when seen as a function of $\y$ is maximized at the correct $\y$ for a given input  $\z$. Given $N$ training pairs $\{(\z_n, \y_n)\}_{n = 1}^N$, $\w$ is learned by solving
\begin{subequations}\label{eq:ssvmprimal} 
\begin{align} 
\minimize_{\w, \bm{\xi} } ~&~  \frac{\lambda}{2} \| \w \|^2 + \frac{1}{N} \sum_{n=1}^{N} \xi_n \label{eq:ssvmprimala}\\
\subjectto ~&~ \langle \w, \bm \psi_n(\tilde{\y}) \rangle \geq L_n(\tilde{\y}) -\xi_n  \label{eq:ssvmconst}
\\~&~~ \forall \tilde{\y} \in \mathcal{Y}_n,~ \forall n\in \mathcal{N} \nonumber
\end{align} 
\end{subequations}
where $\mathcal{N} : = \{1, \ldots, N\}$,  $\bm \psi_n (\tilde{\y} ) := \bm{\phi}(\z_n, \y_n) -\bm{\phi}(\z_n, \tilde{\y}) $, $L_n(\tilde{\y})$ is the incurred loss by predicting $\tilde{\y}$ instead of the given label $\y_n$, $\{\xi_n\}_{n=1}^N$ are slack variables, $\lambda$ is a nonnegative~constant, and $\mathcal{Y}_n$ is the set of all possible outputs for input $\z_n$. In the OCR example, $\mathcal{Y}_n = \{1,\ldots,26\}^{M_n}$, where $M_n$ is the number of characters of the $n$-th word.

Problem~\eqref{eq:ssvmprimal} is difficult since the number of constraints explodes with $|\mathcal{Y}_n|$. If ${\beta}_n(\tilde{\y})$ is the Lagrange dual variable associated with~\eqref{eq:ssvmconst}, vector $\bm{\beta}_n$ is formed with entries $\{\beta_n(\tilde{\y})\}_{\tilde{\y} \in \mathcal{Y}_n}$, and vector $\bm{\beta}$ has entries $\{\bm{\beta}_n\}_{ n\in \mathcal{N}_b}$, the dual of~\eqref{eq:ssvmprimal}~is~\cite{lacoste2013block}
\begin{align} \label{eq:dualssvm}
\minimize_{\bm{\beta}  \in \mathbb{R}^m \atop \bm{\beta}  \geq 0} ~~&~~ f(\bm{\beta}) := \frac{\lambda}{2} \|\A \bm{\beta} \|^2-\b^\top \bm{\beta} \\
\subjectto~~&~~ \mathbf{1}^\top \bm{\beta}_n = 1,\quad \forall n \in \mathcal{N} \nonumber
\end{align} 
where $m : = \sum_{n} |\mathcal{Y}_n|$, $\A \in \mathbb{R}^{d\times m}$ is formed with columns $\{ \frac{1}{\lambda N}\bm \psi_n(\tilde{\y}) \in \mathbb{R}^d |~\tilde{\y} \in \mathcal{Y}_n, n \in \mathcal{N} \} $, and vector $\b \in \mathbb{R}^m$ has entries~$ \{ \frac{1}{N} L_n (\tilde{\y}) \}_{\tilde{\y} \in \mathcal{Y}_n, n \in \mathcal{N}}  $.

A randomized single-block FW is adopted by~\cite{lacoste2013block}, to solve \eqref{eq:dualssvm}. Extending this approach to $B>1$ yields Algorithm~\ref{alg:dsvm}. To avoid storing the    high-dimensional vector $\bm \beta^t$, auxiliary variables 
$\tilde{\w}^t :=  \A \bm{\beta}^t,\quad t =0,1,\ldots$
are introduced. It can be shown that iterates $\{\tilde{\w}^t\}_{t=0,1\ldots}$ converge to the global minimizer of~\eqref{eq:ssvmprimal}; see \cite{lacoste2013block} for additional details.

\begin{algorithm}[t]
\caption{Structural SVMs solver }\label{alg:dsvm}
\begin{algorithmic}[1]
\renewcommand{\algorithmicrequire}{\textbf{Input:}}
\renewcommand{\algorithmicensure}{\textbf{Output:}}
\Require $\{(\z_n, \y_n)\}_{n = 1}^N$,  $\{\mathcal{Y}_n\}_{n=1}^N$, and $B$ 
\State Initialize  $ \bm{\beta}^0$,  $\hat{\ell}^0 =\ell_1^0= \ldots =\ell_{N_b}^0= 0 $, and $t=0$
\State Calculate $\tilde{\w}^0=\tilde{\w}_1^0= \ldots =\tilde{\w}_{N_b}^0 = \A \bm{\beta}^0$
\While{stopping\_criterion not met} 
\State Randomly pick $\mathcal{B}_t\subseteq \mathcal{N}$ such that $|\mathcal{B}_t|= B$
\For{$ n \in \mathcal{B}_t $ } 
\State Compute
 \begin{equation*}
  \mathbf{y}_n^*:= \arg\max_{\y \in \mathcal{Y}_n} ~ L_n(\y) - \langle \tilde{\w}^t, \bm{\psi}_n(\y) \rangle
 \end{equation*}
\State Update $\tilde{\w}_n^{t+1}=(1-\gamma_t)\tilde{\w}_n^t+\frac{\gamma_t}{\lambda N} \bm{\psi}_n (\y_n^*) $
\State Update $\ell_n^{t+1}=(1-\gamma_t)\ell_n^t+ \frac{\gamma_t}{N} L_n (\y_n^*)$
\State Update $\tilde{\w}^{t+1}= \tilde{\w}^{t}+ \tilde{\w}_n^{t+1} -\tilde{\w}_n^{t}$
\State Update $\ell^{t+1}= \ell^{t}+ \ell_n^{t+1} - \ell_n^{t}$
\EndFor
\State $t \leftarrow t+1$
\EndWhile
\end{algorithmic}
\end{algorithm}

\section{Simulated Tests}\label{sec:tests}
This section demonstrates the efficacy of the novel step sizes, and our parallel RB-FW solvers, in the context of the large-scale applications of Sec.~\ref{sec:app}. 

\subsection{Coordination of EV charging} \label{subsec:simua1}
In the first experiment, 63 EVs with maximum charging power $\bar p_n=3.45$~kW $\forall n$, were scheduled. 
 The simulation comprises $T=96$ time slots ranging from 12:00 pm to 12:00~pm of the next day. The values of 
 $\{\mathcal{T}_n\}_{n=1}^N$ and $\{R_n\}_{n=1}^N$ were set according to real travel data of the  National Household Travel Survey~\cite{Federalhighway, liang16scalable}. The base load $\{D_\tau\}_{\tau=1}^T$ were obtained
 by averaging the 2014 residential load data from Southern California Edison~\cite{SCE}. 

Convergence is assessed in terms of the relative error  $\epsilon(\p^t) := \left ({f(\p^t)-f(\p^*)} \right )/{f(\p^*)}$,
where $\p*$ is obtained using the off-the-shelf solver SeDuMi.

The following step-size sequences were compared.
\begin{align} \label{eq:test_stepsizes}
(\text{S1}):~~~\gamma_t &:= \frac{2}{\alpha t+2}  \\
(\text{S2}):~~~\gamma_{t} &:= \frac{\sqrt{\alpha^2 \gamma_{t-1}^4+4\gamma_{t-1}^2} - \alpha\gamma_{t-1}^2}{2},~\gamma_0 = 1 \nonumber\\
(\text{S3}):~~~ \gamma_t &:= \frac{2}{0.5\alpha t+2} \nonumber\\
(\text{S4}):~~~ \gamma_t &:= \frac{2}{0.5\alpha t^{0.9}+2} \nonumber\\
(\text{S5}):~~~ \gamma_t &:= \frac{2}{0.5\alpha t^{0.8}+2} .\nonumber
\end{align}
S2 is the sequence in~\eqref{eq:recursesto}, whereas S1 and S3-S5 are special cases of~\eqref{eq:slowstepsto}. Sequences S1-S5 cover a wide range of decay rates.  S2 vanishes faster than S1 [cf.~\eqref{eq:storecstepa}], whereas  the decay rates of S3-S5 are  smaller than that of S1. Note that S1 boils down to \textcolor{black}{the step size in~\eqref{eq:bcfwgamma}} when setting $B=1$.  
For all $n =1,\ldots, N$, $\p_n^0$ was initialized as 
\begin{equation*}
p_n^{0}(\tau)=\left\{
\begin{array}{ll}
\bar{p}_n(\tau),&\tau=1,\ldots,\bar{\tau}_n^0-1\\
R_n-\sum_{j=1}^{\bar{\tau}_n^0}\bar{p}_n(j),&\tau=\bar{\tau}_n^0\\
0,&\tau=\bar{\tau}_n^0+1,\ldots,T
\end{array}\right.
\end{equation*}
where  the index $\bar{\tau}_n^0 \geq 1$ was found as
\begin{equation*}
\sum_{\tau=1}^{\bar{\tau}_n^0-1}\bar{p}_n(\tau)\leq R_n~\text{and}~\sum_{\tau=1}^{\bar{\tau}_n^0}\bar{p}_n(\tau^0)> R_n.
\end{equation*}

The first experiment assumed that only one vehicle was randomly selected to update its charging profile per iteration. Algorithm~\ref{alg:docsev} with $B=1$ was run with the step sizes S1-S5 for 1,000 iterations. Fig.~\ref{fig:ev1b} depicts the evolution of 
$\epsilon(\p^t)$ across the iteration  index $t$ for Algorithm~\ref{alg:docsev} with  step sizes S1-S5 when $B=1$. It is observed that the algorithm converges towards a global minimum for all the tested step sizes.  In this scenario, the more slowly the step size diminishes, the faster the relative error decreases. Since $B=1$ and Algorithm~\ref{alg:docsev} is a special instance of Algorithm~\ref{alg:bcfw}, Fig.~\ref{fig:ev1b} therefore highlights how randomized single-block FW can benefit from the proposed step sizes. Specifically, the proposed step sizes S3-S5 lead to a much faster convergence than S1, which coincides with the step size in~\textcolor{black}{\eqref{eq:bcfwgamma}} since $B=1$. 

\begin{figure}[t]
   \centering
   \includegraphics[scale=0.6]{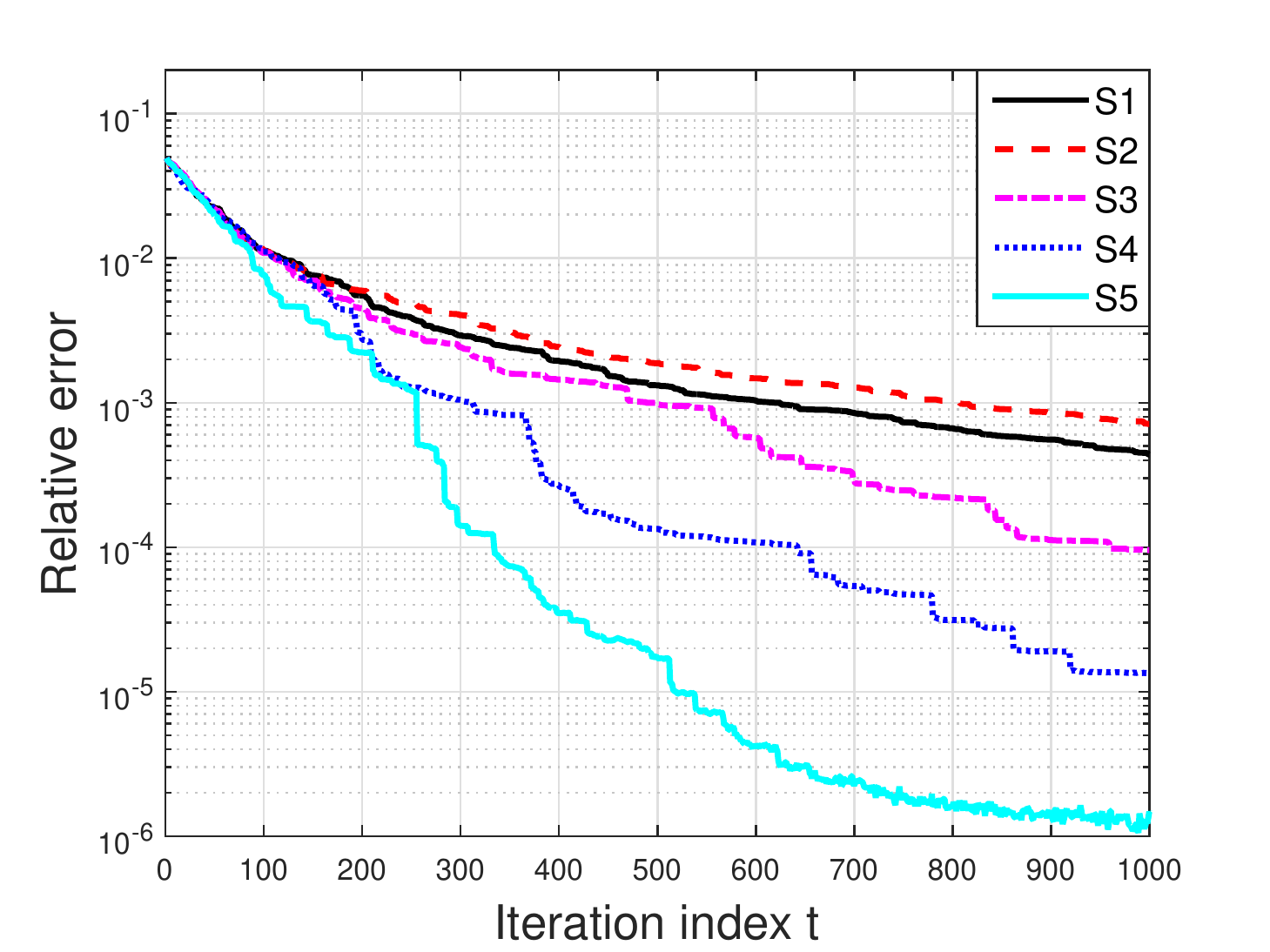}
   \caption{Convergence performance of Algorithm~\ref{alg:docsev} with $B=1$.}\label{fig:ev1b}
\end{figure} 

\begin{figure}[t]
	\centering
	\includegraphics[scale=0.6]{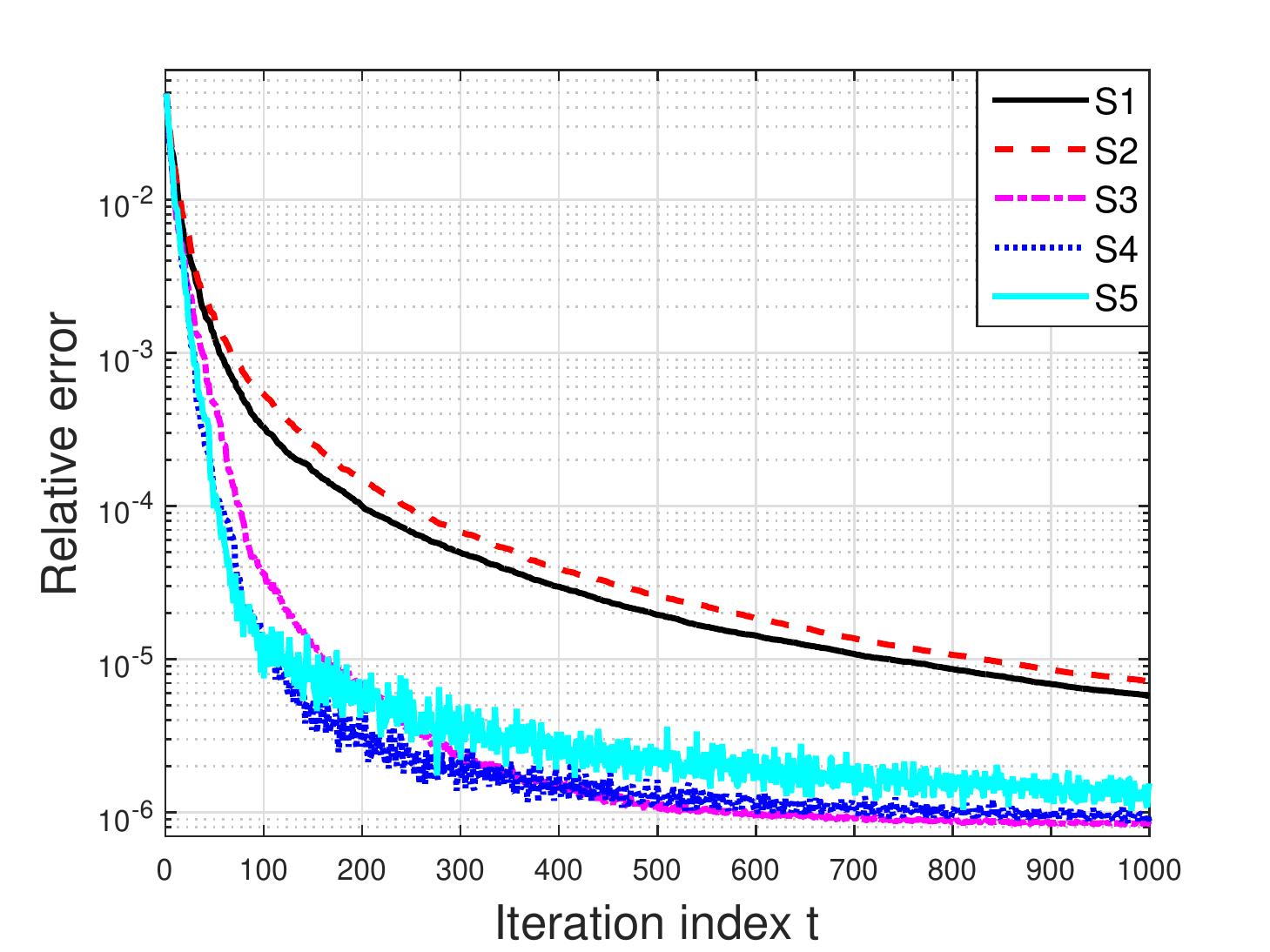}
	\caption{Convergence performance of Algorithm~\ref{alg:docsev} with $B=10$.}\label{fig:ev10b}
\end{figure} 

The second experiment tested Algorithm~\ref{alg:docsev} with $B=10$.  Fig.~\ref{fig:ev10b} further confirms that slowly diminishing step sizes lead to fast convergence in the first few iterations. However, as the iterates approach a minimum, the slowly diminishing step sizes yield larger oscillations; see e.g. S5 in Fig.~\ref{fig:ev10b}. This phenomenon has already been described in Remark~\ref{remark:1}. Comparing Figs.~\ref{fig:ev1b} and~\ref{fig:ev10b} reveals that considerably less iterations are required to achieve a target accuracy for larger~$B$. 
For example, about one fifth  of iterations are now required for Algorithm~\ref{alg:docsev} with S5 to reach $\epsilon(\p)\leq 10^{-5}$. Thus, if the ten blocks can be processed in parallel, setting $B=10$  roughly reduces the computation time by a factor of five, which further corroborates the merits of parallel~RB-FW. 

\begin{figure}[t]
	\centering
	\includegraphics[scale=0.6]{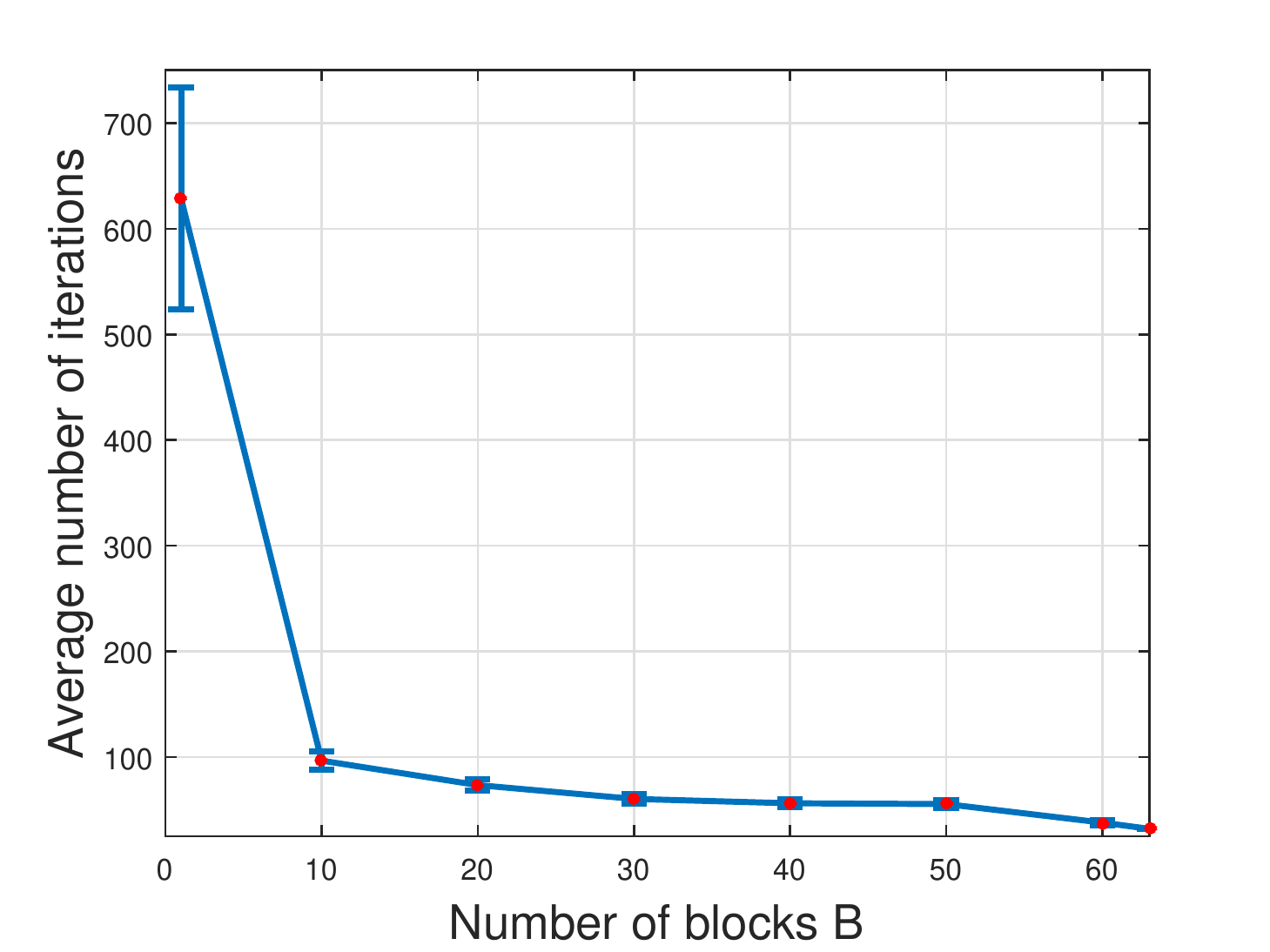}
	\caption{Number of iterations to achieve $\epsilon(\p^t)\leq 10^{-5}$.}\label{fig:blockvsIter}
\end{figure}

\begin{figure}[t]
	\centering
	\includegraphics[scale=0.6]{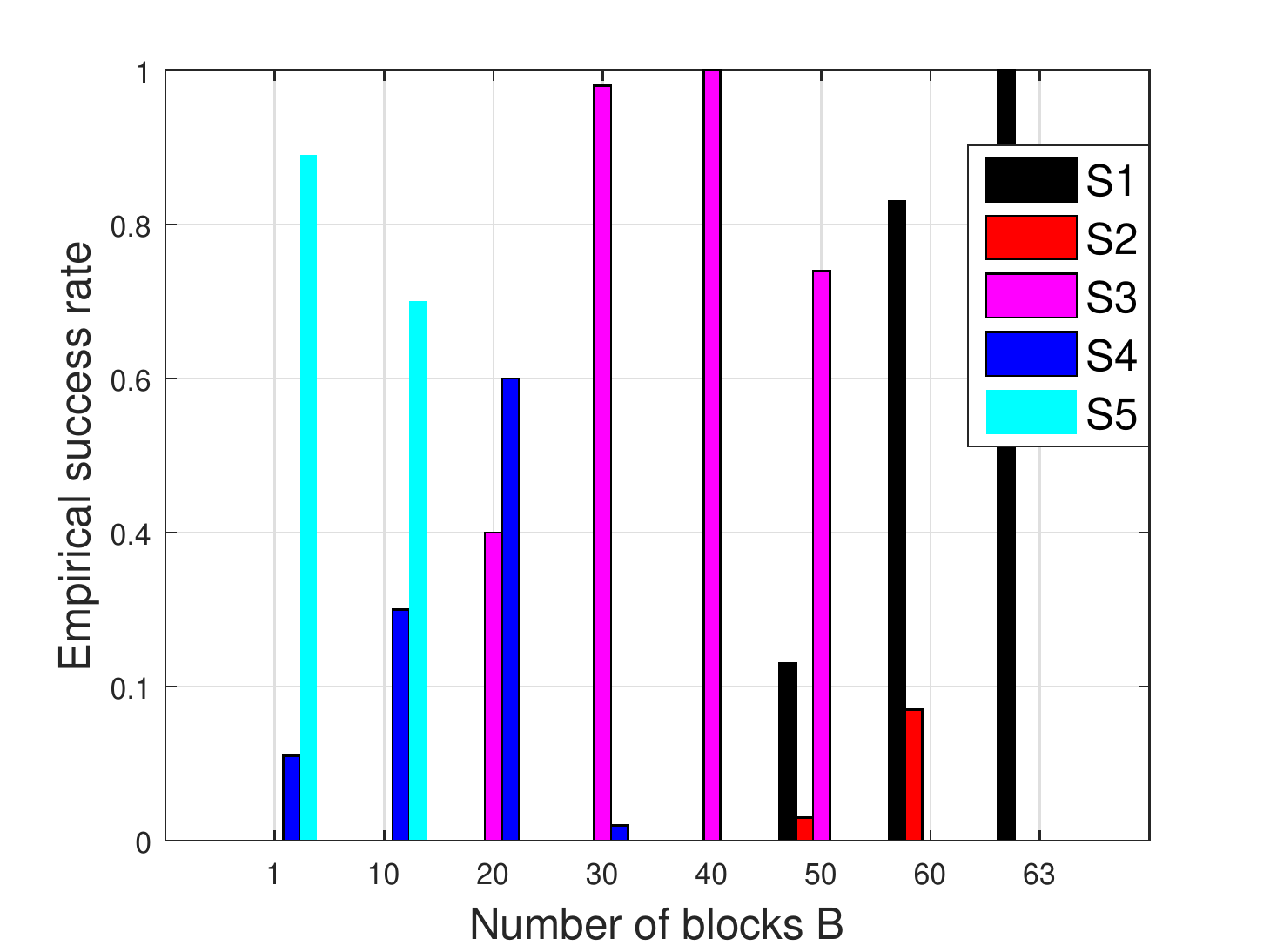}
	\caption{\textcolor{black}{Empirical success rate for  S1-S5 with different values of $B$.}}\label{fig:rate}
\end{figure}

The next experiment highlights the impact of $B$ on the convergence of Algorithm~\ref{alg:docsev}. Five copies of Algorithm~\ref{alg:docsev}, each one with a different step size S1-S5, are executed for 100 independent trials. The minimum value of $t$ such that at least one of these copies satisfies  $\epsilon(\p^t)\leq 10^{-5}$ is recorded. Fig.~\ref{fig:blockvsIter} represents the sample mean and standard derivation of this minimum $t$ averaged over the 100 trials for different values of~$B$. It is observed that both mean and standard derivation decrease for increasing $B$. If each iteration of Algorithm~\ref{alg:docsev} is run in $B$  cores in parallel, then the number of iterations constitutes a proxy for runtime.  Fig.~\ref{fig:blockvsIter} adopts this proxy to showcase the  benefit of adopting $B>1$. Nonetheless, observe that the influence of $B$ on the number of iterations decreases for large~$B$.  
Fig.~\ref{fig:rate} depicts the fraction of trials that each copy of Algorithm~\ref{alg:docsev} is the first among the five copies in achieving $\epsilon(\p^t)\leq 10^{-5}$. This figure reveals that slowly diminishing step sizes are preferable for small values~of~$B$.


\begin{figure}[t]
	\centering
	\includegraphics[scale=0.6]{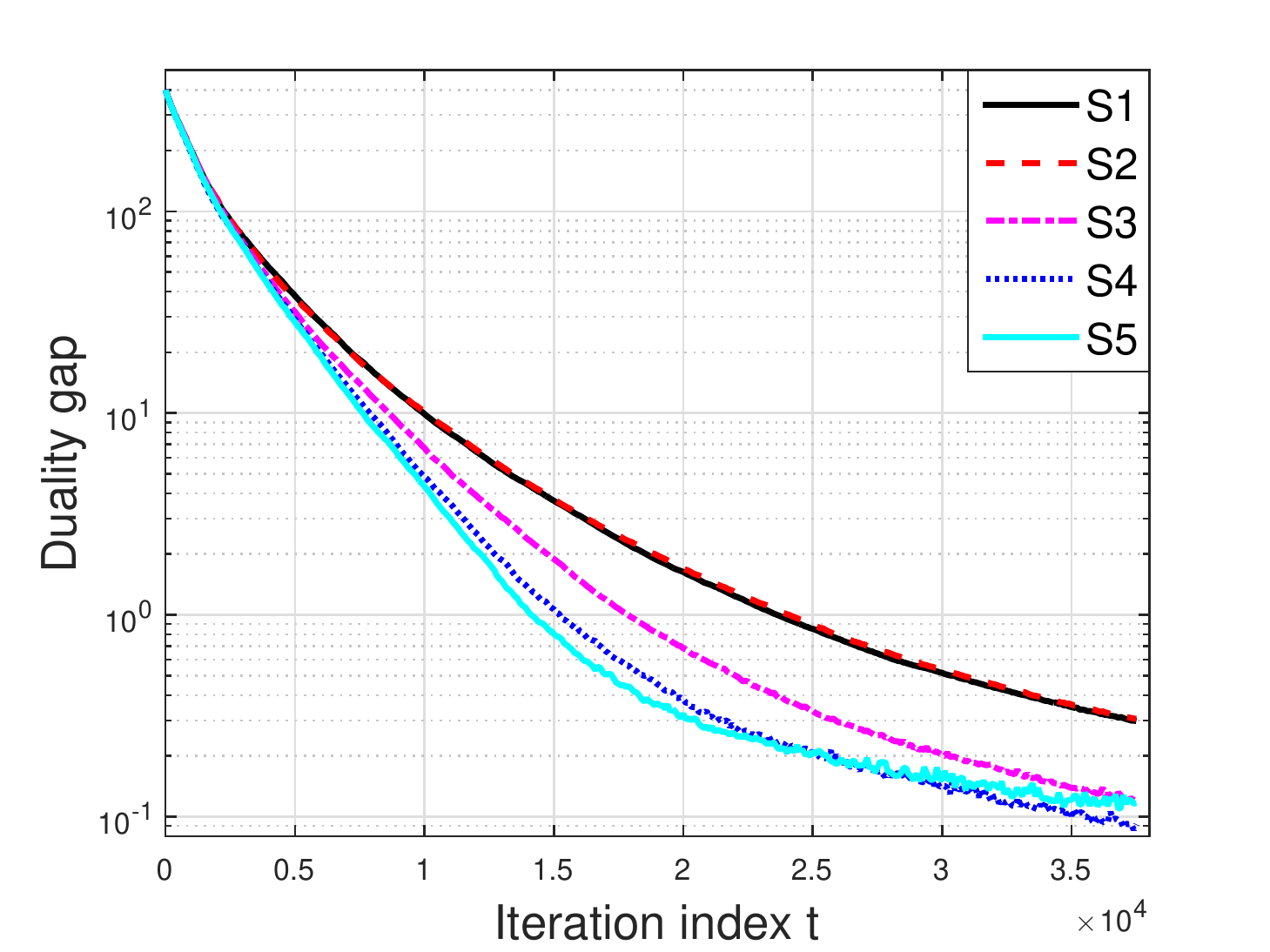}
	\caption{Progress of  $g(\bm{\beta}^t)$ for Algorithm~\ref{alg:dsvm} with $B=1$.}\label{fig:svm1b}
\end{figure}

\begin{figure}[t]
	\centering
	\includegraphics[scale=0.6]{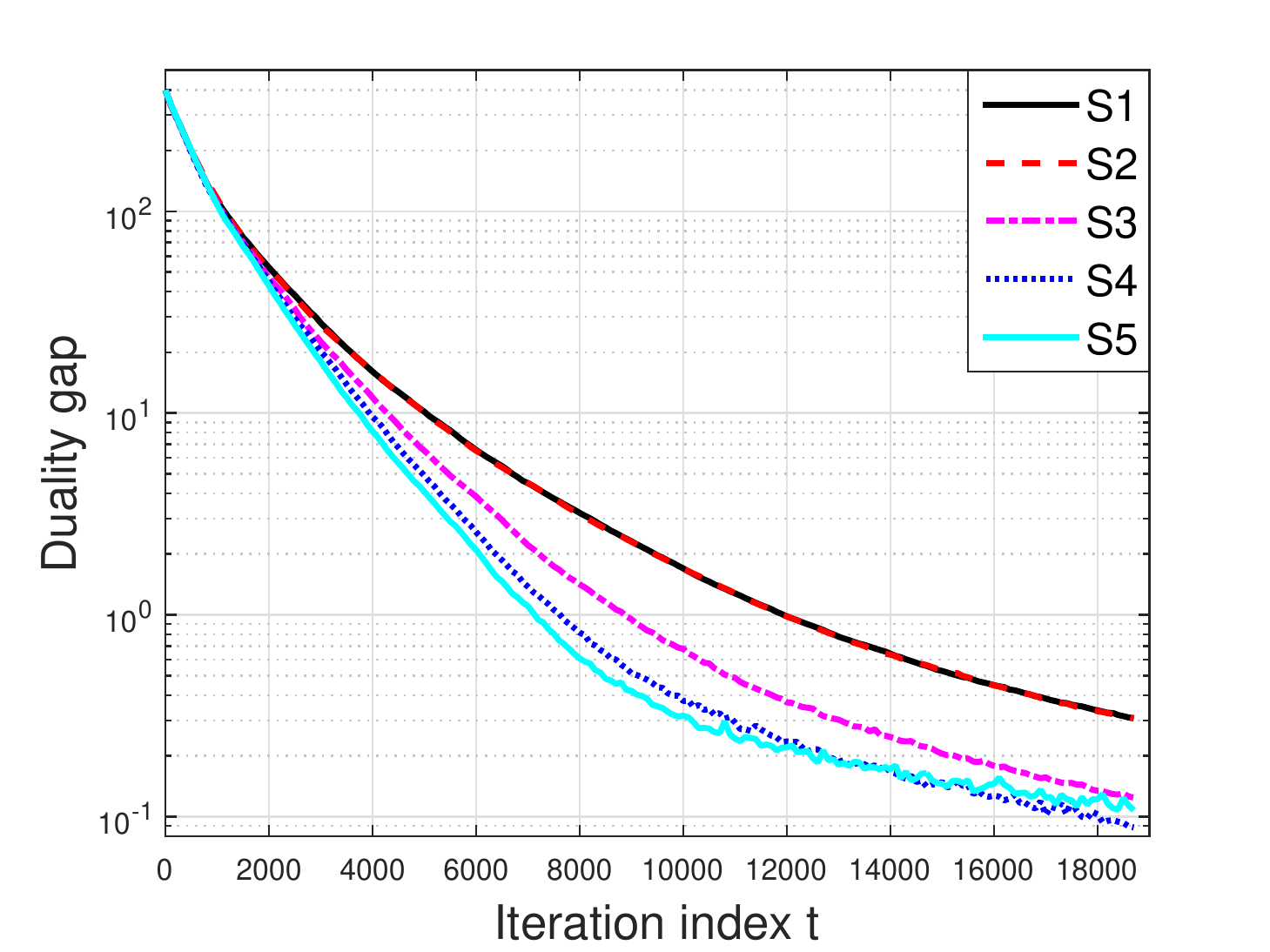}
	\caption{Progress of $g(\bm{\beta}^t)$ for Algorithm~\ref{alg:dsvm} with $B=2$.}\label{fig:svm2b}
\end{figure} 


\subsection{Structural SVMs}\label{subsec:simua1ncev}
The structural SVMs experiment was conducted on a subset of the OCR dataset~\cite{taskar2003max},~\cite{ocr_dataset}.   The  feature mapping  $\bm{\phi}(\z, \y)$, loss function $L_n( \tilde{\y})$, and solution to the subproblems in step 5 of Algorithm~\ref{alg:dsvm} were evaluated using the open source code~\cite{bcfwsolver} released by the authors in~\cite{lacoste2013block}. The dimension of   $\bm{\phi}(\z, \y)$ is $d = 4,028$, and the number of training examples is $N=6,251$.  To initialize each $\bm\beta_n^0$,  one of its entries chosen uniformly at random was set to one, whereas all the remaining  entries were set to zero.  Algorithm~\ref{alg:dsvm} with $\lambda=0.1$ and step sizes S1-S5 was run for six passes through all the training examples. The duality gap $g(\bm \beta^t)$ in~\eqref{eq:gxk0} is depicted in Figs.~\ref{fig:svm1b} and~\ref{fig:svm2b} for $B=1$ and $B=2$, respectively. In both cases,  Algorithm~\ref{alg:dsvm} with S5  outperforms all  other variants in the first few iterations. Furthermore, it can be seen that the required number of iterations to achieve a target accuracy almost halves when increasing $B$ from one~to~two.  

\section{Concluding Summary}\label{sec:conclusions}
The RB-FW algorithm is especially suited for solving high-dimensional constrained learning problems whose feasible set is block separable. For convex programs, the present contribution developed a rich family of feasibility-ensuring step sizes that enable parallel updates of provably convergent RB-FW iterates. The novel step sizes admit various decay rates, leading to flexible  convergence rates of RB-FW. Convergence of RB-FW is further established for constrained nonconvex problems too. Numerical tests using real-world datasets corroborated the speed-up advantage of parallel RB-FW  with the proposed step sizes over randomized single-block FW.  In addition, single-block FW with the developed slowly  diminishing step sizes converges markedly faster than that with existing step sizes. 


\appendix\label{sec:appendix}

\subsection{Proof of Lemma~\ref{lem:declem}} \label{app:A}
Using~\eqref{eq:setcurv} together with steps 4 and 5 of Algorithm~\ref{alg:bcfw}, we find  
	\begin{align*} 
	f(\x^{t+1}) &\leq  f(\x^{t}) +\sum_{n\in\mathcal{B}_t}\langle \x_n^{t+1} -\x_n^t, \nabla_{\x_n} f(\x^t)\rangle + {\gamma_t^2 C_f^{\mathcal{B}_t}}/{2} \nonumber\\
	&=f(\x^{t}) +\sum_{n\in\mathcal{B}_t} \gamma_t\langle \s_n^{t} -\x_n^t, \nabla_{\x_n} f(\x^t)\rangle + \gamma_t^2 C_f^{\mathcal{B}_t}/{2}.
	\end{align*}
	Subtracting $f(\x^*)$ from both sides yields
	\begin{equation*}\label{eq:dec5}
	h(\x^{t+1})\leq h(\x^{t}) +\sum_{n\in\mathcal{B}_t} \gamma_t\langle \s_n^{t} -\x_n^t, \nabla_{\x_n} f(\x^t)\rangle + {\gamma_t^2 C_f^{\mathcal{B}_t}}/{2}.
	\end{equation*}
	Taking conditional expectation with respect to $\mathcal{B}_t$,  we arrive for a given $\x^t$ at  
	\begin{align}
	&\mathbb E_{\mathcal{B}_t}\left[ h(\x^{t+1})| \x^t \right]  \nonumber \\ 
	&\leq h(\x^{t}) +\alpha \sum_{n\in\mathcal{N}_b} \gamma_t\langle \s_n^{t} -\x_n^t, \nabla_{\x_n} f(\x^t)\rangle + {\gamma_t^2 \bar C_f^{B}}/{2} \nonumber \\
	&=  h(\x^{t}) +\alpha \gamma_t \langle\s^t-\x^t , \nabla f(\x^t) \rangle+{\gamma_t^2 \bar C_f^{B}}/{2}  \nonumber \\
	&= h(\x^t) -\alpha \gamma_t g(\x^t)+{\gamma_t^2 \bar C_f^{B}}/{2}  \label{eq:predec} 
	\end{align}
	where the last equality follows from~\eqref{eq:gxk0} and step 4 of Algorithm~\ref{alg:bcfw}.
    Since $\x^t$ is determined by $\{\mathcal{B}_\tau\}_{\tau=0}^{t-1}$, taking expectations in~\eqref{eq:predec}
    with respect to $\{\mathcal{B}_\tau\}_{\tau=0}^{t-1}$ yields~\eqref{eq:dec}.

\subsection{Proof of Lemma~\ref{prop:slowstepsto}} \label{app:C}
	Plugging~\eqref{eq:slowstepsto} into the left-hand side of~\eqref{eq:qualconb} yields  
	\begin{align}\label{eq:slowstep}
	\frac{1-\alpha\gamma_{t+1}}{\gamma_{t+1} ^2 } & = \frac{\left[q(t+1)^{\rho} +2 -\alpha  \right]^2 -\alpha^2}{4} \nonumber \\
	& \leq \frac{\left[q(t+1)^{\rho} +2 -\alpha  \right]^2}{4} \nonumber \\
	& \leq \frac{\left[q(t+1)^{\rho}  -q +2  \right]^2 }{4} .
	\end{align}
	where the  last inequality follows from $q\leq \alpha \leq 1$.
	Consider the auxiliary function  
	$\varphi(x) := (x+c)^\rho - x^\rho -c,\quad x\geq 0$
    for some constant $c \geq 1$, and its first-order derivative  
	\begin{equation*}
	\varphi'(x)  = \rho (x+1)^{\rho -1} - \rho x^{\rho -1}, \quad x\geq 0.
	\end{equation*}
	Since $\rho \leq 1$, it holds that $\varphi'(x) \le 0$, and thus,
	\begin{equation*}
	\varphi(x)  \leq  \varphi(0) = c^\rho -c \leq 0, \quad \forall x\geq 0
	\end{equation*} 
	or, 
	\begin{equation} \label{eq:rhoinequal}
	(x+c)^\rho - c \leq x^\rho, \quad \forall x\geq 0 .
	\end{equation}
	Multiplying both sides of \eqref{eq:rhoinequal} by $q$, and  setting $c=1$ and $x=t$ gives rise to 
	\begin{equation} \label{eq:q}
	0 \leq q(t+1)^{\rho}  -q\leq q t^\rho, \quad \forall t\geq 0.
	\end{equation}
	Combining~\eqref{eq:slowstep} and \eqref{eq:q} yields
	\begin{equation} 
	\frac{1-\alpha\gamma_{t+1}}{\gamma_{t+1} ^2 }   \leq \frac{\left[qt^{\rho}  +2  \right]^2 }{4}  = \frac{1}{\gamma_{t} ^2} \nonumber 
	\end{equation}
	which concludes the proof.

\subsection{Proof of Corollary \ref{cor:slowstep}} \label{app:D}
    Expression \eqref{eq:stoprimalcon} follows directly by  substituting \eqref{eq:slowstepsto} into \eqref{eq:storesult}.	
	To show \eqref{eq:stodual}, apply Theorem~\ref{thm:stodualcon} to verify that  
	\begin{align}
	\alpha {g}_t & \leq \frac{\mathbb{E} \left[h(\x^K)\right]}{\gamma_0^2(t-K+1)\gamma_t}+ \frac{\bar{C}_f^{B}\gamma_K^2}{2\gamma_t} \nonumber \\
	& \leq \frac{(1-\alpha \gamma_0)\gamma_{K-1}^2 h(\x^0)}{\gamma_0^2 (t-K+1)\gamma_t} + \frac{K\bar{C}_f^{B} \gamma_{K-1}^2}{2(t-K+1)\gamma_t} +\frac{\bar{C}_f^{B}\gamma_K^2}{2\gamma_t} \nonumber \\
	& \leq     \frac{\left(1-\alpha\right) \gamma_{K-1}^2 h(\mathbf{x}^{0})}{(t-K+1)\gamma_t} + \frac{\gamma_{K-1}^2 \bar{C}_f^{B}(t+1)}{2\gamma_t(t-K+1)}   \nonumber \\
	& \leq \frac{\gamma_{K-1}^2}{t-K+1}\cdot \frac{(t+1)\bar{C}_f^{B} + 2(1-\alpha)h(\x^0)}{2\gamma_t}  \label{eq:tildeg}
	\end{align}
	for all $K\in \{1, \ldots, t\}$, where  the second inequality stems from~\eqref{eq:storesult} and the third one follows from~$\gamma_K \leq \gamma_{K-1}$ and $\gamma_0 =1$.
	The next step is to bound the first quotient in the right-hand side of \eqref{eq:tildeg}. To this end, set $c=2/q$ and $x= K-1$ in~\eqref{eq:rhoinequal} to deduce that 
	\begin{equation} 
	\gamma_{K-1} = \frac{2}{q (K-1)^\rho +2} \leq \frac{2}{q (K-1+2/q)^\rho } \label{eq:slowstepgammaK}.
	\end{equation}
    Now set  $K =  \lceil \mu (t+2/q) \rceil$, where $\mu$ is an arbitrary constant. Since $K \in \{1,\ldots,t\}$, $\mu$ needs to satisfy 
    \begin{equation}\label{eq:mu}
    0< \mu \leq \frac{t}{t+2/q}.
    \end{equation}
     Since 
	\begin{equation*}
	\lceil \mu (t+2/q) \rceil -1 +2/q \geq \mu (t+2/q) -2 +2/q \geq \mu (t+2/q) > 0
	\end{equation*}
	it follows from~\eqref{eq:slowstepgammaK} that
	\begin{equation*} 
	\gamma_{K-1} \leq \frac{2}{q \mu^\rho (t+2/q)^{\rho}} .
	\end{equation*}
	Therefore,
	\begin{align} 
	\frac{\gamma_{K-1}^2}{t-K+1} \leq &  \frac{4} {q^2(t-K+1) \mu^{2\rho} (t+2/q)^{2\rho}} \nonumber \\   
	\leq &\frac{4} {q^2[t-\mu(t+2/q)]\mu^{2\rho} (t+2/q)^{2\rho}} \label{eq:deno} .
	\end{align}
	Minimizing the right-hand side with respect to $\mu$ in the interval~\eqref{eq:mu} yields
	\begin{equation}\label{eq:K}
	\frac{\gamma_{K-1}^2}{t-K+1} \leq \frac{4(2\rho+1)^{2\rho+1}}{q^2(2\rho)^{2\rho}t^{2\rho+1}} 
    \end{equation}
    for $\mu = \frac{2\rho}{2\rho+1}\frac{t}{t+2/q}$.
	
	From~\eqref{eq:tildeg},  ${g}_t$ can be upper bounded as
	\begin{align*}
	{g}_t \leq &  \frac{\gamma_{K-1}^2}{t-K+1} \cdot\frac{(t+1)\bar{C}_f^{B} + 2(1-\alpha)h(\x^0)}{2\alpha \gamma_t}   \nonumber \\   
	\leq &  \frac{(2\rho+1)^{2\rho+1}(qt^\rho +2) }{\alpha q^2(2\rho)^{2\rho}} \cdot\frac{(t+1)\bar{C}_f^{B} + 2(1-\alpha)h(\x^0)}{t^{2\rho+1}}.
	\end{align*}
    where the second inequality follows from~\eqref{eq:K} and \eqref{eq:slowstepsto}.

\subsection{Proof of Lemma~\ref{prop:faststepsto}} \label{app:E}

	To prove \eqref{eq:storecstepa} by induction, it clearly holds for $t=0$, and assume that it holds also for a fixed $t\geq 0$. Then, one needs to show that
	\begin{equation} \label{eq:gammat}
	\frac{1}{\alpha t+ 1+ \alpha} \leq \gamma_{t+1} \leq \frac{2}{\alpha t+2 + \alpha} .
	\end{equation} 
	To this end, define the auxiliary function
	\begin{equation*}
	\hat  \varphi(x) :=\frac{ \sqrt{\alpha^2 x^4+4x^2} -\alpha x^2 }{2},\quad x \geq 0
	\end{equation*}
	which is monotonically increasing since  
	\textcolor{black}{
	\begin{equation*}
	\hat  \varphi'(x) =\frac{(\alpha^2 x^2+2)- \alpha x\sqrt{\alpha^2x^2+4}}{\sqrt{\alpha^2 x^2+4}}> 0.
	\end{equation*}}
    Thus, by the induction hypothesis we have 
	\begin{equation} \label{eq:gamma}
	\hat  \varphi\left(\frac{1}{\alpha t+1} \right) \leq \gamma_{t+1} = \hat  \varphi(\gamma_{t} ) \leq \hat  \varphi\left(\frac{2}{\alpha t+2}\right).
	\end{equation}
	Note that
	\begin{equation*}
    1-\frac{2\alpha}{\alpha t +2 + \alpha }\leq \left(1-\frac{\alpha}{\alpha t +2 + \alpha }\right)^2= \left(\frac{\alpha t +2}{\alpha t +2 + \alpha }\right)^2
    \end{equation*}    or, equivalently,
	\begin{align} 
	1 &\leq \frac{2\alpha}{\alpha t +2 + \alpha } + \left(\frac{\alpha t +2}{\alpha t +2 + \alpha }\right)^2 \nonumber\\
	& =  \left(\frac{\alpha}{\alpha t+2} + \frac{\alpha t +2}{\alpha t +2 + \alpha}  \right)^2- \frac{\alpha^2}{(\alpha t+2)^2}  \label{eq:ineqrecstep}.
	\end{align}
	This inequality implies that
	\begin{align}
	\hat  \varphi\left(\frac{2}{\alpha t+2}\right) &= \frac{1}{\alpha t +2} \sqrt{\frac{4\alpha^2}{(\alpha t + 2)^2} + 4} - \frac{2\alpha }{(\alpha t +2 )^2}\nonumber  \\
	&\leq  \frac{2}{\alpha t +2 + \alpha}. \label{eq:d1}
	\end{align} 
	Combining \eqref{eq:d1} with the second inequality in \eqref{eq:gamma} proves the second inequality in~\eqref{eq:gammat}. 
	
	On the other hand, since $
	(\alpha t+1+\alpha )^2 \geq\alpha( \alpha t+1+\alpha)  + (\alpha t+1)^2$,
	it holds that
	\begin{align}
	\frac{\alpha^2}{(\alpha t+1)^2}+4 &\geq \frac{\alpha^2}{(\alpha t+1)^2} + \frac{4\alpha}{\alpha t+1+ \alpha } +\frac{4(\alpha  t+1)^2}{(\alpha t+1 + \alpha )^2}\nonumber \\
	& = \left(\frac{\alpha}{\alpha  t+1}+ \frac{2(\alpha t+1)}{\alpha t+1+\alpha } \right)^2.
	\end{align}
	Thus,
	\begin{align}
	\hat  \varphi\left(\frac{1}{\alpha t+1}\right) &= \frac{1}{2(\alpha t+1)}\sqrt{\frac{\alpha^2}{(\alpha  t+1)^2}+4}- \frac{\alpha}{2( \alpha  t+1)^2} \nonumber \\
	&\geq \frac{1}{\alpha t+1 + \alpha }. \label{eq:d2}
	\end{align}
	Combining \eqref{eq:d2} with the first inequality in \eqref{eq:gamma} proves the first inequality in~\eqref{eq:gammat}, thus concluding the proof of~\eqref{eq:storecstepa}.
	
	To prove~\eqref{eq:storecstepb}, one can also proceed by induction. First, $\gamma_1 \leq \gamma_0$ since $\frac{\sqrt{\alpha^2+4}-\alpha}{2}\leq 1$. Assuming $\gamma_{t-1} \leq \gamma_t$, it follows that $\gamma_t\leq \gamma_{t+1}$ since $\hat{\varphi}(x)$ is nondecreasing.

\subsection{Proof of Corollary \ref{cor:recf}} \label{app:F}

	Inequality \eqref{eq:fastfcon} readily follows from \eqref{eq:storesult} and \eqref{eq:storecstepa}. To prove \eqref{eq:fastgcon},  note that \eqref{eq:tildeg} holds because of $\gamma_0 =1 $ and \eqref{eq:storecstep}.  Meanwhile, by the second inequality in~\eqref{eq:storecstepa},~the step size in~\eqref{eq:recursesto} satisfies \eqref{eq:K} for $q=\alpha$ and $\rho=1$, that is 
	\begin{equation} \label{eq:recK}
	\frac{\gamma_{K-1}^2}{t-K+1} \leq \frac{27}{\alpha^2 t^3}.
	\end{equation}
	Plugging~\eqref{eq:storecstepa} together with \eqref{eq:recK} into  \eqref{eq:tildeg}, yields  \eqref{eq:fastgcon}.


\bibliographystyle{IEEEtran}
\bibliography{myabrv,power}

\end{document}